\documentclass[11pt,reqno]{amsart}
\usepackage{amssymb,latexsym,amsmath,amsthm}
\usepackage[left=1in,right=1in]{geometry}
\usepackage{color}
\usepackage{cite}
\usepackage{graphicx}
\usepackage{subfig}
\usepackage{epstopdf}
\DeclareGraphicsExtensions{.eps,.ps,.eps.gz,.ps.gz,.eps.Z}
\usepackage{mathrsfs}
\usepackage{enumerate}
\usepackage{hyperref}
\usepackage{cleveref}
\usepackage[titletoc,title]{appendix}
\usepackage{tikz}
\usetikzlibrary{shapes,shadows,arrows,topaths}
\newtheorem{thm}{Theorem}[section]

\newtheorem{rem}[thm]{Remark}

\title[Stability Analysis of Prey-Predator Model with Infection and Vaccination in Prey]{Stability Analysis of Prey-Predator Model with Infection, Migration and Vaccination in Prey}
\author[\textbf{Sachin Kumar} and \textbf{Harsha Kharbanda}]
{\textbf{Sachin Kumar} \and \textbf{Harsha Kharbanda}\\}
\address{Sachin Kumar,Assistant Professor, Faculty of Mathematical Sciences,
Department of Mathematics,
University of Delhi,
New Delhi-110007,
India.}
\email{sachinambariya@gmail.com}
\address{Harsha Kharbanda,Research Scholar, Faculty of Mathematical Sciences,
Department of Mathematics,
University of Delhi,
New Delhi-110007,
India.}
\email{ssdn0112@gmail.com}
\numberwithin{equation}{section}
\begin{document}
\maketitle
\begin{abstract}
A four dimensional ecoepidemiological model consisting of susceptible prey, infected prey, vaccinated prey and predator is formulated and analyzed in the present work. The functional response is assumed to be of Lotka-Volterra type. We studied systematically the behavior of the model with and without disease in prey. We analyzed mathematically the dynamics of the system such as boundedness of the solutions, existence and stability conditions of equilibria. The basic reproduction number $\mathcal{R}_0$ for the proposed model is computed. Disease is endemic if $\mathcal{R}_0>1$. Numerical simulations are also carried out for the analytical results.
\end{abstract}

\section{\textbf{Introduction}}
To study the dynamic behavior of a model, mathematical modeling is used as an effective tool to describe and analyze the model. In 1798, the British Economist Malthus formulated a single species model \cite{MALTHUS} and subsequently modified by Verhulst. Lotka  and Volterra \cite{{LOTKA},{VOLTERRA}} initially proposed the prey-predator model. Afterwards, prey-predator model became an important research area in applied mathematics. Mathematical epidemiology has become an interesting topic of research since the  model of Kermack-McKendrick \cite{KERMACK} on SIRS (susceptible-infected-removed-susceptible) systems. In 1994, Venturino \cite{EV} discussed the influence of diseases on Lotka-Volterra systems. Many authors have studied prey-predator model and published papers in literature, for example (see \cite{{HADELER},{HAQUE},{HETHCOTE},{JC},{KUANG}}, etc.). Here we focus on the influence of infectious disease on prey-predator interactions. In \cite{HADELER}, Hadeler and Freedman developed and analyzed a prey-predator model with parasitic infection in both species. Kuang and Beretta \cite{KUANG} considered the global behaviors of solutions of a ratio-dependent prey-predator system. Chattopadhyay and Orino \cite{JC} proposed and analyzed a three dimensional predator-prey model with disease only in prey population. In \cite{HAQUE}, Haque and Venturino analyzed the prey-predator model by considering a Holling-Tanner functional response. They also investigated some bifurcations around the disease free equilibrium. A predator-prey model with logistic growth in the prey is modified to include an SIS parasitic infection in the prey studied by Hethcote\ et al. \cite{HETHCOTE}. In \cite{B AND R}, Mukhopadhyaya and Bhattacharyya considered a prey-predator model with Holling type II functional response and observed the dynamics of the system with the effect of diffusion and delay. They also discussed the role of diffusivity on the stability and persistence of the model. Venturino \cite{EV1} investigated the long term behavior in predator-prey model assuming that epidemics occured in prey population and can be transmitted by the contact of predators. \par
Mathematical ecology and mathematical epidemiology are two different fields in the study of biology and applied mathematics. The combination of these two is studied which is termed as eco-epidemiology. Many authors have studied eco-epidemiological models and considered infection in prey population only. Hu and Li \cite{HU AND LI} proposed and analyzed a three dimensional predator-prey delayed model with infection in prey species. They also determined the direction of Hopf bifurcations and the stability of bifurcated periodic solutions. Johri\ et al. \cite{JOHRI} considered a Lotka-Volterra type prey-predator model with disease in prey and analyzed local and global stability. In \cite{JANA}, Jana and  Kar considered a prey-predator model with disease in prey and they used the normal form method and center manifold theorem to investigate the direction of the Hopf bifurcation and stability of the bifurcating limit cycle. Many authors have studied eco-epidemiological models and considered infection in both species such as Kant and Kumar \cite{KANT} formulated and studied a predator-prey model with migrating prey and disease infection in both species. \par

Recently, many authors have proposed and discussed eco-epidemiological models with some assumptions (for instance, \cite{{SARWARDI},{NAJI},{LIU1},{RAHMAN},{XIE},{SINHA},{SILVA}}). They all considered prey-predator model with infection in prey population only. Naji and Mustafa \cite{NAJI} discussed the dynamics of an eco-epidemiological model with nonlinear incidence rate. Silva \cite{SILVA} described the existence of periodic solutions for periodic eco-epidemic models with disease in the prey. Xie\ et al. \cite{XIE} considered the impulsive predator-prey model with communicable disease. The predator-prey model in polluted environment is analyzed by Sinha\ et al. \cite{SINHA}. To explore more about the dynamical systems, one may refer \cite{{KOT},{SMALE},{PERKO}}.

Further, vaccination is important for the elimination of infectious diseases. A vaccine is a biological preparation which provides active acquired immunity to a particular disease. It has been an effective way to reduce disease burden, and is a key tool in maintaining health and welfare. Vaccination is given to all the species including human population. Animal vaccines are part of a category of animal medicines known as veterinary biologics. Vaccines continue to play an increasingly vital role in preventative health and disease control programmes in animals. Vaccination helps to lower the number of infected individuals in the population. \par 

Also, migration is an important demographic event which is found in all the species. The physical movement from one place to another is termed as migration. One of the reasons for animal migration is due to the change in season. For example, bird migration is the regular seasonal movement, often north and south along a flyway, between breeding and wintering grounds and the timing of migration seems to be controlled primarily by changes in day length. The reasons for migration depend on species to species. Since we have taken prey-predator model so scientifically, the effect of migration must be taken into consideration while formulating the mathematical model of the prey-predator systems. Dingle and Drake \cite{DINGLE} explained the term migration for different species. They recognized migration as an adaptation to resources that fluctuate spatiotemporally either seasonally or less predictably. Some authors have studied predator-prey model by taking migration in prey species. For example, Kant and Kumar \cite{KANT} analyzed eco-epidemiological model with infection in both species and migration only in prey population. \par

In the present study, motivated by Hu and Li \cite{HU AND LI}, Liu\ et al. \cite{LIU} and Kant and Kumar \cite{KANT}, we proposed a four dimensional eco-epidemiological model with infection, migration and vaccination in prey population. It consists of susceptible prey, infected prey, vaccinated prey and predator. Local stability has been analyzed. The detailed assumptions for the model is described in the next section. \par

Remaining part of the paper is organized as follows: Section \ref{model} is related to model formulation, Section \ref{results} describes the boundedness of the system and the computation of basic reproduction number $\mathcal{R}_0$. In Section \ref{mwd}, we analyze the model in the absence of infection. In Section \ref{stability}, we discuss the existence and stability conditions of equilibrium points of the main model. Section \ref{numericals} deals with an example to explore analytical results numerically. Paper is concluded in Section \ref{discussion} with a detailed discussion on equilibria of model, role of vaccination and effect of migration.

\section{\textbf{Mathematical Model}}\label{model}
\subsection{Model Formulation}

\noindent
Our model consists of two populations, namely,
the prey, whose population density is denoted by $N(t)$ and
the predator, whose population density is denoted by $P(t)$, where $t$ is the time variable.
We make the following assumptions to formulate our model:
\begin{enumerate}
\item[H(1)] The prey population grows according to logistic law with growth rate $r(r>0)$ and carrying capacity $k(k>0)$ in the absence of disease, vaccination and predation. Therefore we have:
$$\frac{dS}{dt}= r S \left(1-\frac{S}{k}\right).$$
\item[H(2)] Vaccinated prey has a separate class $(V)$ and it is assumed that vaccination is given to only healthy prey with rate of vaccination $\phi$ and $\theta$ is the rate at which the vaccinated individuals return to susceptible class.
\item[H(3)] The prey population is divided into three classes in the presence of disease and vaccination, namely susceptible prey $S(t)$, infected prey $I(t)$ and vaccinated prey $V(t)$, and hence the total prey population at time t will be:
$$N(t)= S(t) + I(t) + V(t).$$
Further, it is assumed that only the susceptible prey can reproduce reaching to its carrying capacity. However, the infected prey does not grow, recover and reproduce.
\item[H(4)] It is assumed that the disease spreads among the prey population only and the transmission of disease between susceptible and infected prey follow the simple law of mass action $ \beta SI$, where $\beta$ is the force of infection.
\item[H(5)] The vaccinated prey still have the possibility of infection with a disease transmission rate $\sigma $ while contacting with infected individuals. $\sigma $ may be assumed to be less than $\beta$ because the vaccinating prey may have some partial immunity during the process or they may recognize the transmission characters of the disease and hence decrease the effective contacts with infected individuals.
\item[H(6)] Predators get the same reward out of predating on healthy, infected and vaccinated prey with different search efficiencies denoted by $p_1, p_2$ and $p_3$, respectively. Also, infected prey become less active and therefore they could get caught easily by the predator compared to healthy prey. Thus, we assume that searching coefficient of the predator for infected prey is greater than that of healthy prey.
\item[H(7)] The functional response of the predator to the prey is assumed to be of Lotka-Volterra type.
\item[H(8)] It is assumed that coefficients of conversing of healthy, infected and vaccinated prey to predator are different denoted by $q_1, q_2$ and $q_3$, respectively.
\item[H(9)] Prey population has migration rates as $m_1, m_2$ and $m_3$ corresponding to healthy, infected and vaccinated prey. It is a natural factor that healthy prey are more strong as compared to infected prey and therefore the probability of migration of healthy prey is more than that of infected prey.
\item[H(10)] It is assumed that all the four species may have different natural death rates.
\end{enumerate}
\tikzstyle{start} = [rectangle, rounded corners, minimum width=1.5cm, minimum height=1.5cm,text centered, draw=black, fill=red!20]

\tikzstyle{arrow} = [thick,->,>=stealth]
\tikzstyle{mytext} = [text width=0.5cm,text centered]

\begin{figure} 
\begin{tikzpicture}
\draw (0,0) node(s) [start] {S} (4,0) node(i) [start] {I} (8,0) node(p) [start] {P} (4,-3) node(v) [start] {V};
\draw [arrow] (s)--(0,2)-|(p.north)node[mytext,midway,xshift=-4cm,anchor=south]{$p_1$};
\draw [arrow] (s)--(v)node[mytext,midway,anchor=south]{$\phi$};
\draw [arrow] (s)--(i) node[mytext,midway,anchor=south]{$\beta$};
\draw [arrow] (i)--(p)node[mytext,midway,anchor=south]{$p_2$};
\draw [arrow] (v)--(p)node[mytext,midway,anchor=west,xshift=0.2cm]{$p_3$};
\draw [arrow] (v.west)--(s.south) node[mytext,midway,anchor=north,xshift=-0.4cm,yshift=0.2cm]{$\theta$};
\draw [arrow] (-2.5,0)--(s)node[mytext,midway,anchor=south]{$r$};
\draw [arrow][dashed](s)--(-2,1.7)node[mytext,yshift=0.2cm,xshift=-0.5]{$m_1$};
\draw [arrow][dashed] (s)--(-1,2)node[mytext,yshift=0.2cm]{$d_1$};
\draw [arrow] (v)--(i)node[mytext,midway,anchor=west]{$\sigma$};
\draw [arrow][dashed] (p)--(10,0)node[mytext,midway,anchor=south]{$d_4$};
\draw [arrow][dashed] (v)--(6,-4)node[mytext,midway,anchor=east,yshift=-0.3cm,xshift=0.3cm]{$d_3$};
\draw [arrow][dashed] (v)--(6,-3)node[mytext,midway,anchor=south,xshift=0.7cm]{$m_3$};
\draw [arrow][dashed] (i)--(5,1.4)node[mytext,midway,anchor=south,yshift=0.2cm,xshift=0.1cm]{$d_2$};
\draw [arrow][dashed] (i)--(4.2,1.4)node[mytext,midway,anchor=south,yshift=0.3cm]{$c$};
\draw [arrow][dashed] (i)--(5.5,1.3)node[mytext,midway,anchor=south,yshift=0.2cm,xshift=0.4cm]{$m_2$};
\end{tikzpicture}
\caption{Schematic diagram of model}
\label{fig:1}
\end{figure}
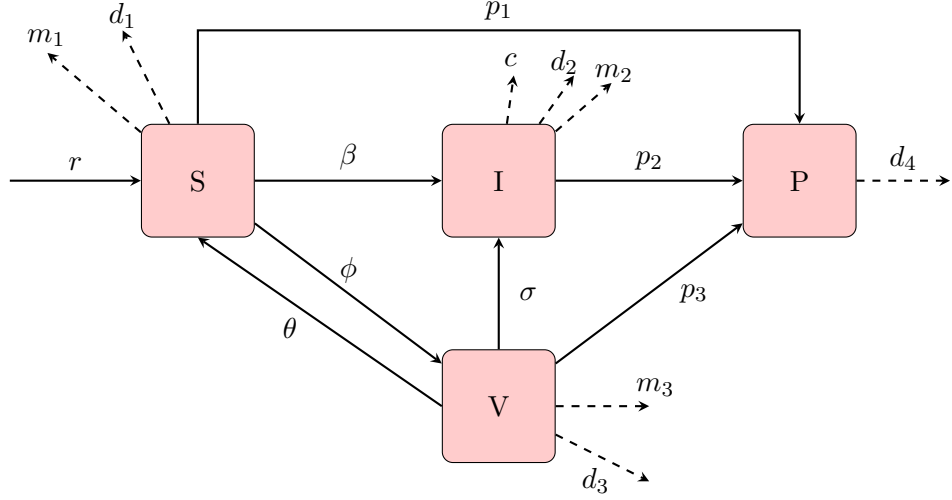

\vspace{12pt}
The mathematical model with above assumptions leads to the following differential equations given by:
\begin{eqnarray} \label{eq1}
\begin{aligned}
\frac{dS}{dt}&=\underbrace{r S\left(1-\frac{S+I}{k}\right)}_{Growth}-\underbrace{\beta S I}_{Infection} -\underbrace{\phi S +\theta V}_{Vaccination} - \underbrace{p_1 P S}_{Predation}- \underbrace{m_1 S}_{Migration} -\underbrace{d_1 S,}_{Mortality}\\
\frac{dI}{dt}&=\underbrace{\beta S I + \sigma V I}_{Infection}- \underbrace{p_2 P I}_{Predation}- \underbrace{m_2 I}_{Migration}-\underbrace{d_2 I - c I,}_{Mortality}\\
\frac{dV}{dt}&=\underbrace{\phi S - \theta V}_{Vaccination}-\underbrace{\sigma V I}_{Infection}-\underbrace{ p_3 P V}_{Predation}-\underbrace{m_3 V}_{Migration} -\underbrace{d_3 V,}_{Mortality}\\
\frac{dP}{dt}&= \underbrace{q_1 p_1 P S + q_2 p_2 P I +q_3 p_3 P V}_{Prey \ Consumption} -\underbrace{d_4 P}_{Mortality}
\end{aligned}
\end{eqnarray}
with initial conditions as $S(0)=S_0 > 0, I(0)=I_0 \geq 0, V(0)=V_0\geq 0$ and $P(0)= P_0 > 0$. All the parameters with their biological/ecological meaning are given in Table \ref{table:1}. The details of the population flux is shown in Figure \ref{fig:1}.
\begin{table}[h!] 
\caption{Biological/ecological meaning of parameters.}
\label{table:1}
\centering
\begin{tabular}{c l}  
\hline Parameter & Biological/ecological meaning \\ \hline
$r$ & Growth rate of prey \\
$\beta$ & Infection coefficient of healthy prey \\
$k$ & Carrying capacity \\
$p_1$ & Healthy prey-predation coefficient\\
$p_2$ & Infected prey-predation coefficient\\
$p_3$ & Vaccinated prey-predation coefficient\\
$q_1$ & Conversion coefficient from healthy prey to predator\\
$q_2$ & Conversion coefficient from infected prey to predator\\
$q_3$ & Conversion coefficient from vaccinated prey to predator\\
$\theta$ & Rate at which vaccination wears off\\
$\phi$ & Rate of Vaccination\\
$\sigma$ & Infection coefficient of vaccinated prey\\
$m_1$ & Migration rate of healthy prey\\
$m_2$ & Migration rate of infected prey\\
$m_3$ & Migration rate of vaccinated prey\\
$d_1$ & Natural death rate of healthy prey\\
$d_2$ & Natural death rate of infected prey\\
$d_3$ & Natural death rate of vaccinated prey\\
$d_4$ & Natural death rate of predator\\
$c$ & Death rate of infected prey due to infection\\
\hline
\end{tabular}
\end{table}
\begin{rem}\label{rem1}
If $\phi = 0$, then there will be no vaccination. Therefore, $ \lim \limits_{t \to \infty} V(t)=0$.
\end{rem}
\begin{rem}\label{rem2}
In this paper, we have maintained difference between mortality and migration but it is interesting to note that migration terms in model \eqref{eq1} look same as mortality terms. 
\end{rem}
\section{\textbf{Preliminary Results}}\label{results}
In this section, we analyze the boundedness of the solutions of the system \eqref{eq1}. Also, the basic reproduction number $\mathcal{R}_0$ is computed for the proposed model.
\subsection{Boundedness}
Since all the parameters are non-negative, the right hand side of \eqref{eq1} is a smooth function of variables $(S,I,V,P)$ in the positive octant,\\ $\Omega =\lbrace(S,I,V,P)| S\geq 0, I \geq 0, V \geq 0, P \geq 0\rbrace$. It is easy to prove that $\Omega$ is an invariant set. Since system \eqref{eq1} is homogeneous, we have $S=0, I=0, V=0$ and $P=0$ is one solution. The uniqueness and existence theorem ensures that any trajectory starting from the first quadrant remains in it, that is, no trajectory will cross the coordinate planes.\\
\noindent
Now we will prove the boundedness of the system \eqref{eq1}.
\begin{thm}
All the solutions of the system \eqref{eq1} are uniformly bounded.
\end{thm}
\begin{proof}
Let $\chi = S + I + V + P$. Its time derivative is given as:
\begin{equation*}
\frac{d\chi}{dt}=\frac{dS}{dt}+ \frac{dI}{dt}+ \frac{dV}{dt}+ \frac{dP}{dt}.
\end{equation*}
Now, for each $\mu>0$, we have
\begin{equation*}
\begin{split}
\frac{d\chi}{dt}+\mu \chi = & \; r S\left(1-\frac{S+I}{k}\right) -p_1 P S -m_1 S -d_1 S - p_2 P I- m_2 I-d_2 I - c I  - p_3 P V- m_3 V \\& -d_3 V+ q_1 p_1 P S + q_2 p_2 P I +q_3 p_3 P V -d_4 P +\mu \chi \\
\frac{d\chi}{dt}+\mu \chi  \leq & \; k \frac{(r+\mu)^2}{4r} -(m_2+d_2+c-\mu)I-(m_3+d_3-\mu) V-(d_4-\mu)P \quad (\text{since} \; q_1,q_2,q_3<1)\\
\leq &\; k \frac{(r+\mu)^2}{4r}=\eta \; \text{if} \; \mu<\text{min}(m_2+d_2+c,m_3+d_3,d_4).
\end{split}
\end{equation*} 
Therefore, we have ${d\chi}/{dt}+\mu \chi \leq \eta.$\\ Now, by applying theory of differential inequality, we obtain \\$0 < \chi(S,I,V,P) < (\eta / \mu)(1-\exp(-\mu t))+ \chi(S_0,I_0,V_0,P_0) \exp(-\mu t)$ and for $t\to \infty$, we have $0 < \chi(S,I,V,P) < (\eta / \mu)$.\\
Hence, all the solutions of the system \eqref{eq1} are confined in the region $\Lambda= \{(S,I,V,P)\in {\mathbb{R}^4_+}:\chi = \eta / \mu + \epsilon \; \text{for any} \; \epsilon>0\}$.
\end{proof}
\subsection{The basic Reproduction number}
The next generation matrix method \cite{DIEKMANN1} is used to calculate the basic reproduction number $\mathcal{R}_0$ \cite{JONES}. Clearly, $I$ is the only relevant class of infection. The class $I(t)$ from our model is
\begin{equation}\label{eq2}
\frac{dI}{dt}=\beta S I+\sigma V I- p_2 P I- m_2 I-d_2 I - c I.
\end{equation}
Therefore, two matrices $F$ and $V$ corresponding to the gain and loss components of equation \eqref{eq2} are defined as $F = (\beta S +\sigma V)$ and $V = (p_2 P + m_2 +d_2 + c )$. These matrices evaluated at the disease-free equilibrium point $E_1(S_1,0,V_1,0)$ where\\
\begin{equation*}
\begin{split}
S_1 &= \frac{k}{r}\left(r-\phi-m_1-d_1+\frac{\theta\phi}{\theta+m_3+d_3}\right) \ \text{and}\\
V_1 &= \frac{\phi k}{r(\theta+m_3+d_3)} \left(r-\phi-m_1-d_1+\frac{\theta\phi}{\theta+m_3+d_3}\right).
\end{split}
\end{equation*}
Now, the next generation matrix is defined as $G = F V^{-1}$. The basic reproduction number is the dominant eigenvalue of the next generation matrix. Thus,
$$ \mathcal{R}_0 = \frac{\beta S_1 +\sigma V_1}{c+m_2 +d_2}.$$
If $\mathcal{R}_0 >1$, then disease is endemic.
\section{\textbf{Model without disease}}\label{mwd}
In this section, model is transformed with the assumption that there does not occur any infection within prey population. Therefore, the model \eqref{eq1} is reduced into three dimensional prey-predator model with vaccination in prey. The model becomes:
\begin{equation}\label{eq3}
\begin{split}
\frac{dS}{dt}&=r S\left(1-\frac{S}{k}\right)-\phi S +\theta V-p_1 P S-m_1 S -d_1 S, \\
\frac{dV}{dt}&=\phi S - \theta V- p_3 P V- m_3 V -d_3 V,\\
\frac{dP}{dt}&= q_1 p_1 P S +q_3 p_3 P V -d_4 P
\end{split}
\end{equation}
with initial conditions $S(0)>0,V(0)\geq 0$ and $P(0)> 0$.
This system \eqref{eq3} has following equilibrium points:
\begin{enumerate}
\item [(i)] Trivial equilibrium, $E^{(0)}=(0,0,0)$.
\item [(ii)] Predator-free equilibrium, $E^{(1)}=(S_1,V_1,0)$, where
\begin{equation*}
\begin{split}
S_1 &= \frac{k}{r}\left(r-\phi-m_1-d_1+\frac{\theta\phi}{\theta+m_3+d_3}\right) \ \text{and}\\
V_1 &= \frac{\phi k}{r(\theta+m_3+d_3)} \left(r-\phi-m_1-d_1+\frac{\theta\phi}{\theta+m_3+d_3}\right).
\end{split}
\end{equation*}
\item [(iii)] Interior equilibrium, $E^{(2)}=(S_2,V_2,P_2)$, where 
\begin{equation}\label{eqn1}
\begin{split}
S_2 &= \frac{(d_4 -p_3 q_3 V_2)}{p_1 q_1},\\
V_2 &= \frac{\phi d_4}{(\theta p_1 q_1 + d_3 p_1 q_1 + m_3 p_1 q_1 + \phi p_3 q_3 + p_1 p_3 q_1 P_2)}
\end{split}
\end{equation}
and $P_2$ is governed by
\begin{eqnarray}\label{eqn2}
\begin{aligned}
& \left(-\frac{r d_4 (\theta + d_3 + m_3 + p_3 P_2)^2}{k}\right) + [r \theta + (r - \phi) (d_3 + m_3) - (\theta + d_3 +m_3)(d_1 + m_1 + p_1 P_2)\\& - P_2 (-r + \phi + d_1 +m_1 + p_1 P_2) p_3][p_1 (\theta + d_3 + m_3+ p_3 P_2) q_1 +\phi p_3 q_3]=0.
\end{aligned}
\end{eqnarray}
\end{enumerate}
\subsection{Existence of equilibria and stability for disease free model}
To analyze the disease free model, we use the variational matrix which is given as:
$$J'=
\begin{pmatrix}
r-\frac{2rS}{k}-\phi-p_1P-m_1-d_1 & \theta & -p_1S\\
\phi & -\theta-p_3P-m_3-d_3 & -p_3V\\
q_1 p_1 P & q_3 p_3 P & q_1 p_1 S+q_3 p_3 V-d_4
\end{pmatrix}.$$
\subsubsection{Trivial equilibrium $(E^{(0)})$}
The trivial equilibrium $(E^{(0)}(0,0,0))$ always exists.
The jacobian matrix evaluated at $(E^{(0)})$ is
$$J'(E^{(0)})=
\begin{pmatrix}
r-\phi-m_1-d_1 & \theta & 0\\
\phi & -\theta-m_3-d_3 & 0\\
0 & 0 & -d_4
\end{pmatrix}.$$
The characteristic polynomial corresponding to $J'(E^{(0)})$ is
\begin{equation}\label{e1}
(-\lambda - d_4) (-\theta \phi - (r- \phi - d_1 - m_1- \lambda) (\lambda + \theta + d_3 + m_3))=0.
\end{equation}
One of the eigenvalues of $J'(E^{(0)})$ is $(-d_4)$ and the remaining two roots of \eqref{e1} will be analyzed by the quadratic equation given as:
\begin{equation*}
 (\lambda-r+ \phi + d_1 + m_1 )(\lambda + \theta + d_3 + m_3)-\theta \phi=0.
\end{equation*}
Now by Routh-Hurwitz criterion, $(E^{(0)})$ is locally stable whenever the following conditions are satisfied:
\begin{equation}\label{c1}
\begin{split}
(-r + \theta + \phi + d_1+ d_3 + m_1 + m_3)>0, \\
(-r+ \phi + d_1 + m_1)(\theta + d_3 + m_3)-\theta \phi>0.
\end{split}
\end{equation}
\subsubsection{Predator-free equilibrium $(E^{(1)})$}
The Predator-free equilibrium $(E^{(1)}(S_1,V_1,0))$ exists when the following condition is satisfied:
\begin{equation*}
\left(r-\phi-d_1-m_1+\frac{\theta \phi }{\theta+d_3+m_3}\right)>0.
\end{equation*}
The jacobian matrix evaluated at $(E^{(1)})$ is
$$J'(E^{(1)})=
\begin{pmatrix}
r-\frac{2rS_1}{k}-\phi-m_1-d_1 & \theta & -p_1S_1\\
\phi & -\theta-m_3-d_3 & -p_3V_1\\
0 & 0 & q_1 p_1 S_1+q_3 p_3 V_1-d_4
\end{pmatrix}.$$
The characteristic equation corresponding to $J'(E^{(1)})$ is
\begin{equation}\label{e2}
\frac{1}{k} (-\lambda - d_4 + p_1 q_1 S_1 + p_3 q_3 V_1)[(2 r S_1 + k (-r + \lambda + \phi+d_1 + m_1)) (\lambda + \theta + d_3 + m_3)-k \theta \phi] =0.
\end{equation}
One eigenvalue of $J'(E^{(1)})$ is $\lambda_1= - d_4 +p_1 q_1 S_1 + p_3 q_3 V_1$ and the remaining two roots of the characteristic equation \eqref{e2} will be given by the quadratic equation written as:
$$ \left( \lambda+\frac{2 r S_1}{k} -r  + \phi + d_1 + m_1\right) (\lambda + \theta + d_3 + m_3)-\theta \phi =0. $$
By using Routh-Hurwitz criterion, $(E^{(1)})$ is locally stable provided the following conditions are satisfied:
\begin{align*}
\left(-r + \frac{2 r S_1}{k} + \theta + \phi + d_1 + d_3 + m_1 + m_3\right)>0,\\
\left(\frac{2 r S_1}{k}-r+\phi+d_1+m_1\right)(\theta+d_3+m_3)-\theta \phi >0 
\end{align*}
and $\lambda_1= (- d_4 +p_1 q_1 S_1 + p_3 q_3 V_1)<0$, where $S_1>0$ and $V_1>0$.
\subsubsection{Interior equilibrium $(E^{(2)})$} \label{inteq}
The interior equilibrium $(E^{(2)}(S_2,V_2,P_2))$ exists if the following conditions are satisfied:
$$(d_4 -p_3 q_3 V_2)>0 \ \text{and $P_2$ is the positive root of equation \eqref{eqn2}}.$$
The jacobian matrix corresponding to interior equilibrium is
$$J'(E^{(2)})=
\begin{pmatrix}
\alpha_{11} &  \alpha_{12} & \alpha_{13}\\
\alpha_{21} &  \alpha_{22} & \alpha_{23}\\
\alpha_{31} &  \alpha_{32} & \alpha_{33}
\end{pmatrix}$$
where \\
$ \alpha_{11} = r-\frac{2 r S_2}{k} -\phi -p_1 P_2 - m_1-d_1, \quad \alpha_{12}= \theta, \hspace{1.48in} \alpha_{13} =  -p_1 S_2, \\
\alpha_{21}= \phi, \hspace{2.08in} \alpha_{22} = -\theta - p_3 P_2 -m_3-d_3, \quad \alpha_{23}= -p_3 V_2, \\
\alpha_{31}= q_1 p_1 P_2, \hspace{1.73in} \alpha_{32}=q_3 p_3 P_2,\hspace{1.12in}  \alpha_{33} = q_1 p_1 S_2+ q_3 p_3 V_2-d_4.$\\
\newline
The characteristic equation of the above matrix is given by:
\begin{equation} \label{eq7}
\lambda^3 + B_1 \lambda^2 + B_2 \lambda + B_3=0,
\end{equation}
where
\begin{align*}
B_1 &= -tr(A) = -(\alpha_{11} + \alpha_{22} + \alpha_{33}),\\
B_2 &= \text{Sum of the second order principal minors}\\ &=
(\alpha_{11}\alpha_{22} - \alpha_{12}\alpha_{21})+(\alpha_{11}\alpha_{33}-\alpha_{13}\alpha_{31})+(\alpha_{22}\alpha_{33}-\alpha_{23}\alpha_{32}),\\
B_3 &= -det(A) = -[\alpha_{11}(\alpha_{22} \alpha_{33} -\alpha_{32} \alpha_{23})+\alpha_{12}(\alpha_{31} \alpha_{23}-\alpha_{33} \alpha_{21}) + \alpha_{13}(\alpha_{21}\alpha_{32}-\alpha_{31}\alpha_{22})]
\end{align*}
which can be seen in \cref{Appendix:a1}.\\
Thus, from Routh-Hurwitz criterion, $(E^{(2)})$ is locally stable when the following conditions are satisfied:\\
\begin{equation}\label{c2}
\begin{cases}
B_1 B_2 > B_3,\\
B_i > 0, \ i=1,2,3.
\end{cases}
\end{equation}
 \section{\textbf{Equilibria and their stability of main model}}\label{stability}
The equilibrium points of the system \eqref{eq1} are as follows:
\begin{enumerate}
\item Trivial equilibrium $E_0(0,0,0,0)$.
\item Disease-free equilibrium $E_1(S_1,0,V_1,0)$,
where\begin{equation*}
\begin{split}
S_1 &= \frac{k}{r}\left(r-\phi-m_1-d_1+\frac{\theta\phi}{\theta+m_3+d_3}\right) \ \text{and}\\
V_1 &= \frac{\phi k}{r(\theta+m_3+d_3)} \left(r-\phi-m_1-d_1+\frac{\theta\phi}{\theta+m_3+d_3}\right).
\end{split}
\end{equation*}
\item Equilibrium $E_2(S_2,0,V_2,P_2)$,
where $S_2,V_2$ and $P_2$ are defined by equations \eqref{eqn1} and \eqref{eqn2}, respectively.
\item Equilibrium $E_3(0,I_3,0,P_3)$, where
\begin{equation*}
\begin{aligned}
I_3&= \frac{d_4}{q_2 p_2},\\
P_3&= -\frac{c+d_2+m_2}{p_2}.
\end{aligned}
\end{equation*}
\item Predator-free equilibrium $E_4(S_4,I_4,V_4,0)$, where
\begin{equation*}
\begin{split}
S_4 &= \frac{(c + d_2 + m_2) (\theta  + d_3 + m_3+ \sigma I_4)}{\beta \theta + \sigma \phi + \beta d_3+ \beta m_3+ \beta \sigma I_4},\\
V_4 &= \frac{ \phi(c+d_2+m_2)}{\beta \theta  + \sigma \phi + \beta d_3+ \beta m_3+ \beta \sigma I_4}
\end{split}
\end{equation*}
and $I_4$ is the root of
\begin{align} \label{eqi}
g(I_4)=&\theta \phi- (\theta + d_3 + m_3+ I_4 \sigma)\times \nonumber \\ &\left(\phi + d_1 + m_1+\beta I_4 + \frac{r}{k} 
\left(k - I_4 - \frac{(c + d_2 + m_2) (\theta  +d_3 + m_3+ I_4 \sigma)}{ \sigma \phi+\beta (\theta + d_3+ m_3+ \sigma  I_4)}\right)\right).
\end{align}
\item Interior equilibrium $E_5(S_5,I_5,V_5,P_5)$, where $S_5, I_5, V_5 $ and $P_5$ are defined as:
\begin{equation}
\begin{split}
I_5 &= \frac{d_4 - p_1 q_1 S_5 - p_3 q_3 V_5}{p_2 q_2}, \\
V_5 &= \frac{P}{Q}, \\
P_5 &= \frac{\beta S_5 +\sigma V_5-c- d_2- m_2}{p_2},
\end{split}
\end{equation} 
where
\begin{align}\label{e3}
P =& r d_4 S_5 + k \beta d_4 S_5 -r p_1 q_1 S_5^2 -k \beta p_1 q_1 S_5^2 -c k p_1 q_2 S_5+k \beta p_1 q_2 S_5^2 - k d_2 p_1 q_2 S_5\nonumber\\ &-k m_2 p_1 q_2 S_5 -k r p_2 q_2 S_5 + r p_2 q_2 S_5^2+k \phi p_2 q_2 S_5 +k d_1 p_2 q_2 S_5 +k m_1 p_2 q_2 S_5, \\
Q =& -k \sigma p_1 q_2 S_5+k \theta p_2 q_2 +r p_3 q_3 S_5 + k \beta p_3 q_3 S_5 \nonumber
\end{align}
and $S_5$ is the zero of 
\begin{align}\label{e4}
h(S_5)=& \phi p_2 [k q_2(-\sigma p_1 S_5+ \theta p_2) +(r + k \beta) p_3 q_3 S_5]^2- [(r + k \beta) (d_4 - p_1 q_1 S_5)+ q_2\{-k p_1(c -\beta S_5 \nonumber \\
&+ d_2 + m_2) + p_2(r S_5 + k (-r + \phi) + k (d_1 + m_1))\}][\sigma d_4 (-k S_5 \sigma p_1 + k \theta p_2+ p_3 S_5 (r + k \beta))\nonumber \\
&+ k S_5^2 \sigma^2 p_1^2 q_1 +p_2 q_2(k \theta  p_2(\theta + d_3 + m_3)+ p_3(-c k \theta + S_5 (k \beta \theta + r S_5 \sigma+ k \sigma (-r + \phi))+  \nonumber \\ 
&k S_5 \sigma (d_1 + m_1) -k \theta (d_2 + m_2)))+S_5 p_3 ((r \theta + k \beta \theta + k r \sigma -r S_5 \sigma- k \sigma \phi - k \sigma d_1 + m_1) \nonumber \\
&+ r (d_3 + m_3)+ k \beta (d_3 + m_3)) p_2 - (r + k \beta) (c - S_5 \beta + d_2 + m_2) p_3) q_3+  S_5 \sigma p_1 (k p_2 (-\theta q_1\nonumber \\
&- (\theta + d_3 + m_3) q_2) +p_3(-S_5 (r + k \beta) q_1 + k q_3 (c - S_5 \beta + d_2 + m_2)))].
\end{align}
\end{enumerate}

\subsection{Existence conditions of equilibria}
The existence conditions of equilibrium points are as follows:
\begin{enumerate}
\item [(i)]Trivial equilibrium $(E_0)$ always exists.
\item [(ii)]Equilibrium $(E_1)$ exists if the  following condition is satisfied:
\begin{equation*}
\left(r-\phi-d_1-m_1+\frac{\theta \phi }{\theta+d_3+m_3}\right)>0.
\end{equation*}
\item [(iii)]Equilibrium $(E_2)$ exists whenever the following conditions are satisfied:
\begin{equation*}
(d_4 -p_3 q_3 V_2)>0
\end{equation*}
and $P_2$ is the positive root of equation \eqref{eqn2}. 
\item [(iv)] Equilibrium $(E_3)$ does not exist.
\item [(v)] Equilibrium $(E_4)$ exists when $I_4$ is the positive root of equation \eqref{eqi}.
\item [(vi)] Equilibrium $(E_5)$ exists provided the following conditions are satisfied:
\begin{align*}
&(d_4 - p_1 q_1 S_5 - p_3 q_3 V_5)>0,\\
&(\beta S_5 +\sigma V_5-c- d_2- m_2)>0,
\end{align*}
$S_5$ is the positive root of equation \eqref{e4} and one of these two conditions are satisfied (not simultaneously):
$P > 0, Q < 0$ where $P$ and $Q$ are defined in eq. \eqref{e3}.
\end{enumerate}
\subsection{Stability analysis of equilibria}
To analyze the stability of equilibrium points, we use the jacobian matrix of system \ref{eq1} which is given by
\begin{center}
$J = 
\begin{pmatrix}
A_{11} & -\frac{rS}{k}-\beta S & \theta & -p_1 S\\
\beta I & A_{22} & \sigma I & -p_2 I\\
\phi & -\sigma V & A_{33} & -p_3 V\\
q_1 p_1 P & q_2 p_2 P & q_3 p_3 P & A_{44}
\end{pmatrix},$ \end{center}
where
\begin{equation*}
\begin{split}
A_{11} &= r-2r\frac{S}{k}-\frac{rI}{k}-\beta I -\phi -p_1 P - m_1-d_1,\\
A_{22} &= \beta S + \sigma V -p_2 P - m_2-d_2-c,\\
A_{33} &=-\sigma I -\theta - p_3 P -m_3-d_3,\\
A_{44} &=q_1 p_1 S+q_2 p_2 I + q_3 p_3 V-d_4.
\end{split}
\end{equation*}
\subsubsection{Trivial equilibrium point $(E_0)$}
The jacobian matrix evaluated at $E_0(0,0,0,0)$ is
\begin{center}
$J(E_0) = 
\begin{pmatrix}
r-\phi- m_1-d_1 & 0 & \theta & 0\\
0 & - m_2-d_2-c & 0 & 0\\
\phi & 0 & -\theta  -m_3-d_3 & 0\\
0 & 0 & 0 & -d_4
\end{pmatrix}.$ \end{center}
The characteristic equation corresponding to $J(E_0)$ is
\begin{equation} \label{eq4}
(-\lambda-d_4)(-\lambda-c-d_2-m_2)[-\theta \phi-(r-\lambda-\phi-d_1-m_1)(\lambda+\theta+d_3+m_3)]=0.
\end{equation}
Two eigenvalues of $J(E_0)$ are $-d_4 , - (c+d_2+m_2)$ and the remaining two roots of the characteristic equation \eqref{eq4} will be analyzed by the quadratic equation 
\begin{equation*}
 \lambda^2+A_1 \lambda +A_2=0
\end{equation*}
where,
\begin{equation}\label{c3}
\begin{split}
A_1&=-r + \theta+ \phi + d_1 + d_3 + m_1 + m_3,\\
A_2&=-r \theta + \theta d_1 - r d_3 + \phi d_3+ d_1 d_3 + \theta m_1 + d_3 m_1 - r m_3 + \phi m_3 +d_1 m_3 + m_1 m_3.
\end{split}
\end{equation}
For $A_1>0$ and $A_2>0$, trivial equilibrium $(E_0)$ is locally stable by Routh-Hurwitz criterion.
\subsubsection{Disease-free equilibrium $(E_1)$}
The jacobian matrix corresponding to $E_1(S_1,0,V_1,0)$ will be
$$J(E_1)=
\begin{pmatrix}
r-\frac{2 r S_1}{k}- \phi - d_1- m_1 & -S_1(\frac{r}{k}+\beta) & \theta & -p_1 S_1\\
0 & \beta S_1 + \sigma V_1- m_2-d_2-c & 0 & 0\\
\phi & -\sigma V_1 & -\theta -m_3-d_3 & -p_3 V_1\\
0 & 0 & 0 & q_1 p_1 S_1 + q_3 p_3 V_1-d_4
\end{pmatrix}. $$
The characteristic equation of this matrix is 
\begin{equation}\label{eq5}
\begin{split}
&\frac{1}{k}(-c + \beta S_1 +  \sigma V_1- d_2- m_2- \lambda) (-\lambda - d_4 + p_1 q_1 S_1 + p_3 q_3 V_1)\\&[(2 r S_1 + k (-r + \lambda + \phi+ d_1 + m_1)) (\lambda + \theta + d_3 + m_3)-k \theta \phi] =0.
\end{split}
\end{equation}
Two eigenvalues of $J(E_1)$ are:
\begin{equation}\label{c4}
\begin{split}
\lambda_1 &= -c - d_2- m_2 +\beta S_1+  \sigma V_1= (c+d_2+m_2) (\mathcal{R}_0-1), \\
\lambda_2 &= - d_4 +p_1 q_1 S_1 + p_3 q_3 V_1 
\end{split}
\end{equation}
and the remaining two roots of the equation \eqref{eq5} will be analyzed by solving the quadratic equation
$$ \left( \lambda+\frac{2 r S_1}{k} -r  + \phi + d_1 + m_1\right) (\lambda + \theta + d_3 + m_3)-\theta \phi =0. $$
\noindent
Now, by using Routh-Hurwitz criterion, disease-free equilibrium $(E_1)$ is locally stable whenever the following conditions are satisfied:\\
\begin{align*}
\left(-r + \frac{2 r S_1}{k} + \theta + \phi + d_1 + d_3 + m_1 + m_3\right) >0,\\
\left(\frac{2 r S_1}{k}-r+\phi+d_1+m_1\right)(\theta+d_3+m_3)-\theta \phi >0
\end{align*}
together with the conditions $\lambda_1<0$ $( \text{or} \ \mathcal{R}_0 <1)$ and $\lambda_2<0$, where $S_1>0$ and $V_1>0$.
\subsubsection{Equilibrium $(E_2)$}
The Jacobian matrix evaluated at $E_2(S_2,0,V_2,P_2)$ is given by:
$$J(E_2) = 
\begin{pmatrix}
a_{11} & -\frac{rS_2}{k}-\beta S_2 & \theta & -p_1 S_2\\
0 & a_{22} & 0 & 0\\
\phi & -\sigma V_2 & a_{33} & -p_3 V_2\\
q_1 p_1 P_2 & q_2 p_2 P_2 & q_3 p_3 P_2 & a_{44}
\end{pmatrix},$$
where \\
$ a_{11} = r-\frac{2 r S_2}{k} -\phi -p_1 P_2 - m_1-d_1,\\
a_{22} =\beta S_2 + \sigma V_2 -p_2 P_2 - m_2-d_2-c,\\ a_{33} = -\theta - p_3 P_2 -m_3-d_3,\\
a_{44} = q_1 p_1 S_2+ q_3 p_3 V_2-d_4.$\\
\newline
One eigenvalue of matrix $J(E_2)$ is
$a_{22}=-c  - d_2 - m_2 + \beta S_2 +  \sigma V_2 -p_2 P_2$ and the remaining eigenvalues of this matrix will be given by the eigenvalues of the matrix $A$ defined as
$$A = 
\begin{pmatrix}
r-\frac{2 r S_2}{k} -\phi -p_1 P_2 - m_1-d_1 &  \theta & -p_1 S_2\\
\phi &   -\theta - p_3 P_2 -m_3-d_3 & -p_3 V_2\\
q_1 p_1 P_2 & q_3 p_3 P_2 & q_1 p_1 S_2+ q_3 p_3 V_2-d_4
\end{pmatrix},$$
which is similar to the matrix which we have discussed in subsection \ref{inteq}.
Thus, from Routh-Hurwitz criterion, equilibrium $(E_2)$ is locally stable provided the following conditions are satisfied:
\begin{equation*}
\left\{
\begin{array}{ll}
B_1 B_2 > B_3,\\
B_i > 0, \ i=1,2,3 \; \text{and}\\
(-c + \beta S_2 + \sigma V_2 - d_2 - m_2 -p_2 P_2)<0
\end{array}
\right.
\end{equation*}
where $S_2, V_2$ and $P_2$ are all positive.
\subsubsection{Predator-free equilibrium $(E_4)$}
The jacobian matrix evaluated at equilibrium $E_4(S_4,I_4,V_4,0)$ is given by:
$$J(E_4) = 
\begin{pmatrix}
b_{11} & -\frac{rS_4}{k}-\beta S_4 & \theta & -p_1 S_4\\
\beta I_4 & b_{22} & \sigma I_4 & -p_2 I_4\\
\phi & -\sigma V_4 & b_{33} & -p_3 V_4\\
0 & 0 & 0 & b_{44}
\end{pmatrix},$$
where\\
$
b_{11} = r-2r\frac{S_4}{k}-\frac{rI_4}{k}-\beta I_4 -\phi - m_1-d_1,\\
b_{22} = \beta S_4 + \sigma V_4  - m_2-d_2-c,\\
b_{33} = -\sigma I_4 -\theta -m_3-d_3,\\
b_{44} = q_1 p_1 S_4 +q_2 p_2 I_4 + q_3 p_3 V_4 -d_4.
$\\
\newline
One eigenvalue of $J(E_4)$ is $b_{44} = q_1 p_1 S_4 +q_2 p_2 I_4 + q_3 p_3 V_4 -d_4$ and the remaining three eigenvalues are given by the eigenvalues of the matrix $B$ written as:
$$B=
\begin{pmatrix}
\kappa_{11} &  \kappa_{12} & \kappa_{13}\\
\kappa_{21} &  \kappa_{22} & \kappa_{23}\\
\kappa_{31} & \kappa_{32} & \kappa_{33}
\end{pmatrix}$$
where \\
$ \kappa_{11} = r-2r\frac{S_4}{k}-\frac{rI_4}{k}-\beta I_4 -\phi - m_1-d_1, \quad \quad
\kappa_{12}= -\frac{rS_4}{k}-\beta S_4,\quad \quad \kappa_{13} = \theta,\\
\kappa_{21}= \beta I_4, \quad \quad
\kappa_{22} = \beta S_4 + \sigma V_4  - m_2-d_2-c, \quad \quad
\kappa_{23}= \sigma I_4, \\
\kappa_{31}= \phi, \quad \quad
\kappa_{32}= -\sigma V_4,\quad \quad
\kappa_{33} = -\sigma I_4 -\theta -m_3-d_3.$\\
\newline
The characteristic equation of the above matrix is given by:
\begin{equation} \label{eq8}
\lambda^3 + C_1 \lambda^2 + C_2 \lambda + C_3=0,
\end{equation}
where
\begin{align*}
C_1 &= -tr(B) = -(\kappa_{11} + \kappa_{22} + \kappa_{33}),\\
C_2 &= (\kappa_{11}\kappa_{22} - \kappa_{12}\kappa_{21})+(\kappa_{11}\kappa_{33}-\kappa_{13}\kappa_{31})+(\kappa_{22}\kappa_{33}-\kappa_{23}\kappa_{32}),\\
C_3 &= -det(B)
= -[\kappa_{11}(\kappa_{22} \kappa_{33} -\kappa_{32} \kappa_{23})+\kappa_{12}(\kappa_{31} \kappa_{23}-\kappa_{33} \kappa_{21}) + \kappa_{13}(\kappa_{21}\kappa_{32}-\kappa_{31}\kappa_{22})]
\end{align*}
which can be seen in \cref{Appendix:a2}.\\
Now by using Routh-Hurwitz criterion, predator-free equilibrium $(E_4)$ is locally stable provided the following conditions are satisfied:
\begin{equation}\label{c5}
\left\{
\begin{array}{ll}
C_1 C_2 > C_3,\\
C_i > 0, \ i=1,2,3 \; \text{and}\\
(q_1 p_1 S_4 +q_2 p_2 I_4 + q_3 p_3 V_4 -d_4)<0
\end{array}
\right.
\end{equation}
where $S_4, I_4$ and $V_4$ are all positive.
\subsubsection{Interior Equilibrium $(E_5)$}
The jacobian matrix evaluated at interior equilibrium $E_5(S_5,I_5,V_5,P_5)$ is given by:
\begin{center}
$J(E_5) = 
\begin{pmatrix}
c_{11} & c_{12} & c_{13} & c_{14}\\
c_{21} & c_{22} & c_{23} & c_{24}\\
c_{31} & c_{32} & c_{33} & c_{34} \\
c_{41} & c_{42} & c_{43} & c_{44}
\end{pmatrix},$ \end{center}
where
\begin{align*}
c_{11} &= r-2r\frac{S_5}{k}-\frac{rI_5}{k}-\beta I_5 -\phi -p_1 P_5 - m_1-d_1,\quad
c_{12} = -\left(\frac{r}{k}+\beta\right) S_5, \quad
c_{13} = \theta,\;
c_{14} = -p_1 S_5,\\
c_{21} &= \beta I_5,\quad \quad
c_{22} = \beta S_5 + \sigma V_5 -p_2 P_5 - m_2-d_2-c,\quad \quad
c_{23} = \sigma I_5,\quad \quad
c_{24} = -p_2 I_5,\\
c_{31} &= \phi,\quad \quad
c_{32} = -\sigma V_5,\quad \quad
c_{33} =-\sigma I_5 -\theta - p_3 P_5 -m_3-d_3,\quad \quad
c_{34} = -p_3 V_5,\\
c_{41} &= q_1 p_1 P_5,\quad \quad
c_{42} = q_2 p_2 P_5,\quad \quad
c_{43} = q_3 p_3 P_5,\quad \quad
c_{44} =q_1 p_1 S_5 +q_2 p_2 I_5 + q_3 p_3 V_5 -d_4.
\end{align*}
The characteristic equation of the above matrix is written as:
\begin{equation}\label{eq9}
\lambda^4+D_1 \lambda^3 + D_2 \lambda^2 + D_3 \lambda+ D_4=0
\end{equation}
where\\
$D_1 = -tr(J(E_5)),\\
D_2 = \text{Sum of all the possible second order principal minors},\\ 
D_3 = -(\text{Sum of all the possible third order principal minors}),\\
D_4 = det(J(E_5))$\\
which can be seen in \cref{Appendix:a3}. Therefore, from Routh-Hurwitz criterion, we can conclude that equilibrium $(E_5)$ is stable provided the following conditions are satisfied:
\begin{equation*}
\left\{
\begin{array}{ll}
D_i > 0, \ i=1,2,3,4,\\
D_1 D_2 > D_3 \; \text{and}\\
D_1(D_2 D_3 - D_1 D_4) - D_3^2>0.
\end{array}
\right.
\end{equation*}
\subsection{Global stability of equilibria}
In this section, we will prove the global stability of the equilibrium points for different 2-D planes.
\begin{thm}
$(E_4)$ is globally asymptotically stable in $S-V$ plane.
\end{thm}
\begin{proof}
Let 
\begin{equation*}
\hspace{-3.5in} h_1(S,V)= \frac{1}{SV}
\end{equation*}
It is obvious that $h_1(S,V)>0$ if $S>0$ and $V>0$.\\
Now, we denote
\begin{equation*}
\begin{split}
F_1(S,V)&=r S\left(1-\frac{S+I}{k}\right)-\beta S I-\phi S +\theta V - m_1 S -d_1 S,\\
F_2(S,V)&= \phi S - \theta V-\sigma V I-m_3 V -d_3 V, \\
\Delta(S,V)&= \frac{\partial}{\partial S} [F_1 h_1]+\frac{\partial}{\partial V} [F_2 h_1].
\end{split}
\end{equation*}
Then, \begin{equation*}
\hspace{-2in} \Delta(S,V)= -\frac{r}{kV}-\frac{\theta}{S^2}-\frac{\phi}{V^2}.
\end{equation*}
Thus, $\Delta(S,V)<0$ for all $S>0$ and $V>0$. Therefore, by using Bendixson-Dulac criterion, there will be no periodic orbit in the first quadrant.\\
This completes the proof.
\end{proof}
\begin{rem}
In similar manner, we have observed that the equilibrium points $(E_1)$, $(E_2)$, $(E_4)$ and $(E_5)$  are globally asymptotically stable in different planes as:
\begin{enumerate}
\item[(1)] $E_1$ is globally asymptotically stable in $S-V$ plane.
\item[(2)] $E_2$ is globally asymptotically stable in $S-V$, $S-P$ and $V-P$ planes.
\item[(3)] $E_4$ is globally asymptotically stable in $S-I$ and $I-V$ planes.
\item[(4)] $E_5$ is globally asymptotically stable in $S-I$, $S-V$, $S-P$, $I-V$ and $V-P$ planes.
\end{enumerate}
\end{rem}
\section{\textbf{Numerical simulation}}\label{numericals}
The dynamic behavior of the model around equilibrium points has been seen in previous sections with theoretical results. Now in this section, we have performed some numerical simulations to observe and describe the effect on the dynamics of the system \eqref{eq1}.
For the set of parameters defined by $P_1 = \{r, k, \beta, \phi, \theta, \sigma, c, p_1, p_2, p_3,q_1, q_2, q_3, d_1, d_2, d_3, d_4 \} = \{1.1, 2.9, 1.2, 1.2, 1.2, 0.2, 0.35, 0.125,\\ 0.125, 0.125, 0.75, 0.8, 0.75, 0.25, 0.125, 0.1,  0.25 \}$. 

Some parameter values have been taken from Jana and Kar \cite{JANA}, Hu and Li \cite{HU AND LI} and some are assumed.
In the absence of migration, disease free model \eqref{eq3} with these values will be:
\begin{equation}\label{eq31}
\begin{split}
\frac{dS}{dt}&=(1.1) S\left(1-\frac{S}{2.9}\right)-(1.2) S +(1.2) V-(0.125) P S -(0.25) S, \\
\frac{dV}{dt}&=(1.2) S - (1.2) V- (0.125) P V -(0.1) V,\\
\frac{dP}{dt}&= (0.09375) P S +(0.09375) P V -(0.25) P.
\end{split}
\end{equation}
It has been observed that this system have all the three types of equilibria and they are:\\ $(i) E^{(0)}=(0,0,0), (ii) E^{(1)}=(1.99755,1.84389,0)$ and $(iii) E^{(2)}=(1.44427,1.22239,0.942539)$.
\begin{figure}[h!]
\centering
\includegraphics[scale=0.5]{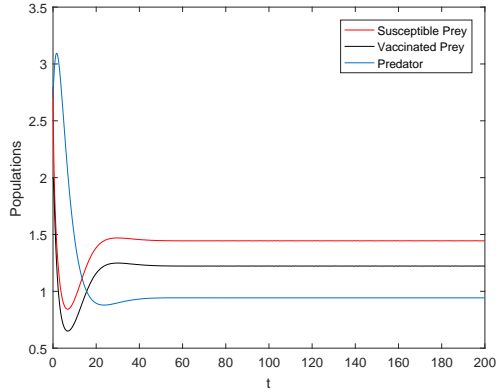}
\caption{Behavior of solutions for the system \eqref{eq31}.}
\label{fig:2}
\end{figure}

Trivial equilibrium $E^{(0)}$ is unstable as one condition of \eqref{c1} is not satisfied, that is, $[(-r+ \phi + d_1 + m_1)(\theta + d_3 + m_3)-\theta \phi]=-0.985 <0$. The predator-free equilibrium $E^{(1)}$ is not stable as one eigenvalue calculated as $-d_4 +p_1 q_1 S_1 + p_3 q_3 V_1=0.110135>0$. Thus, $E^{(1)}$ is also unstable.\\ Now for the interior equilibrium $E^{(2)}$, equation \eqref{eq7} is
\begin{equation}\label{en1}
\lambda^3 + 2.98129 \lambda^2 + 0.806172 \lambda + 0.079073=0.
\end{equation}
It can be seen that all the coefficients of equation \eqref{en1} are positive and the another condition of \eqref{c2} is also satisfied as $(2.98129)(0.806172)= 2.40343 > 0.079073$. Hence, equilibrium $E^{(2)}$ is locally stable. Figures \ref{fig:2}, \ref{fig:3} and \ref{fig:4} give the results corresponding to the system \eqref{eq31}.

\begin{figure}[h!]
\subfloat{\includegraphics[scale=0.35]{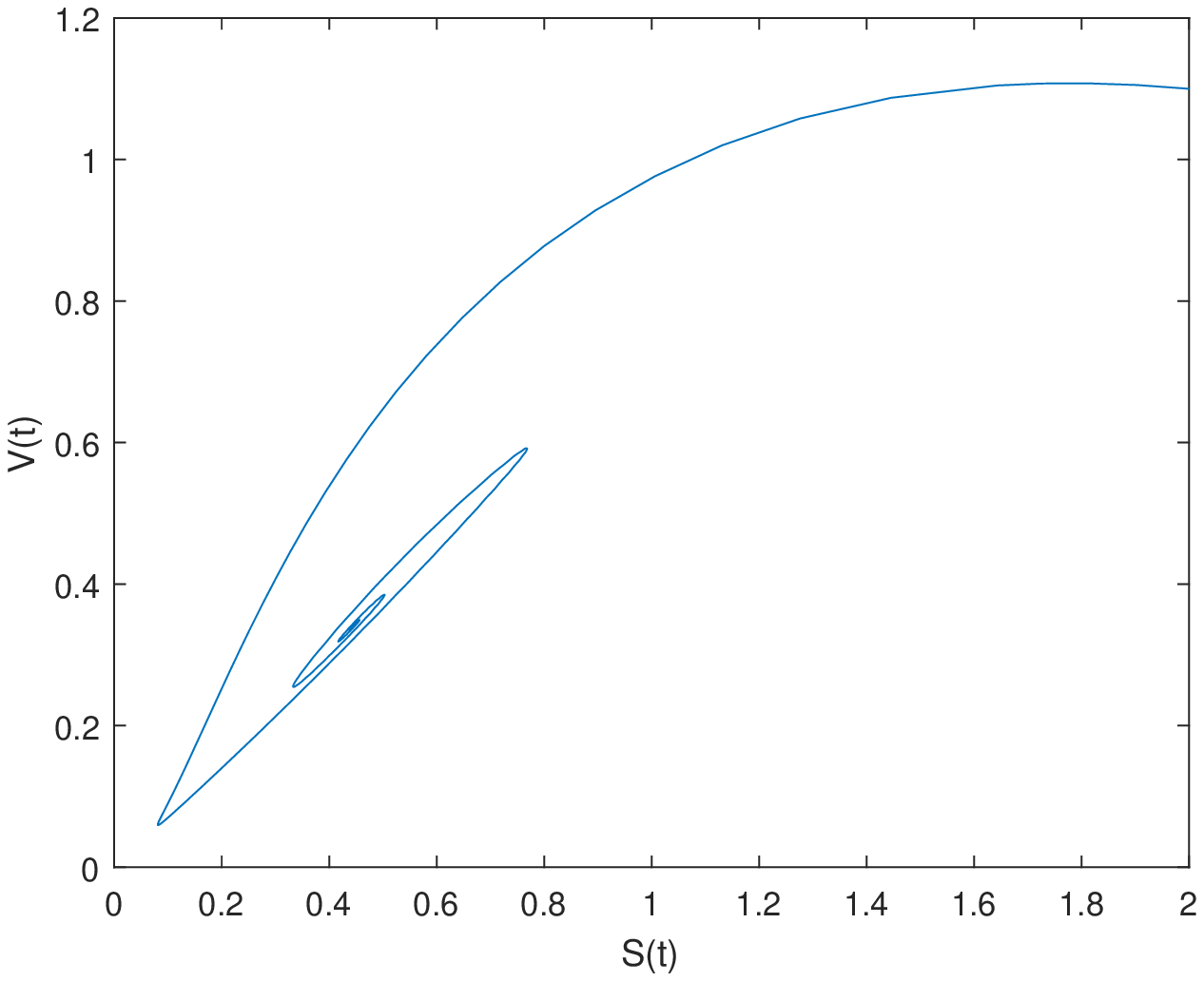}}
\subfloat{\includegraphics[scale=0.35]{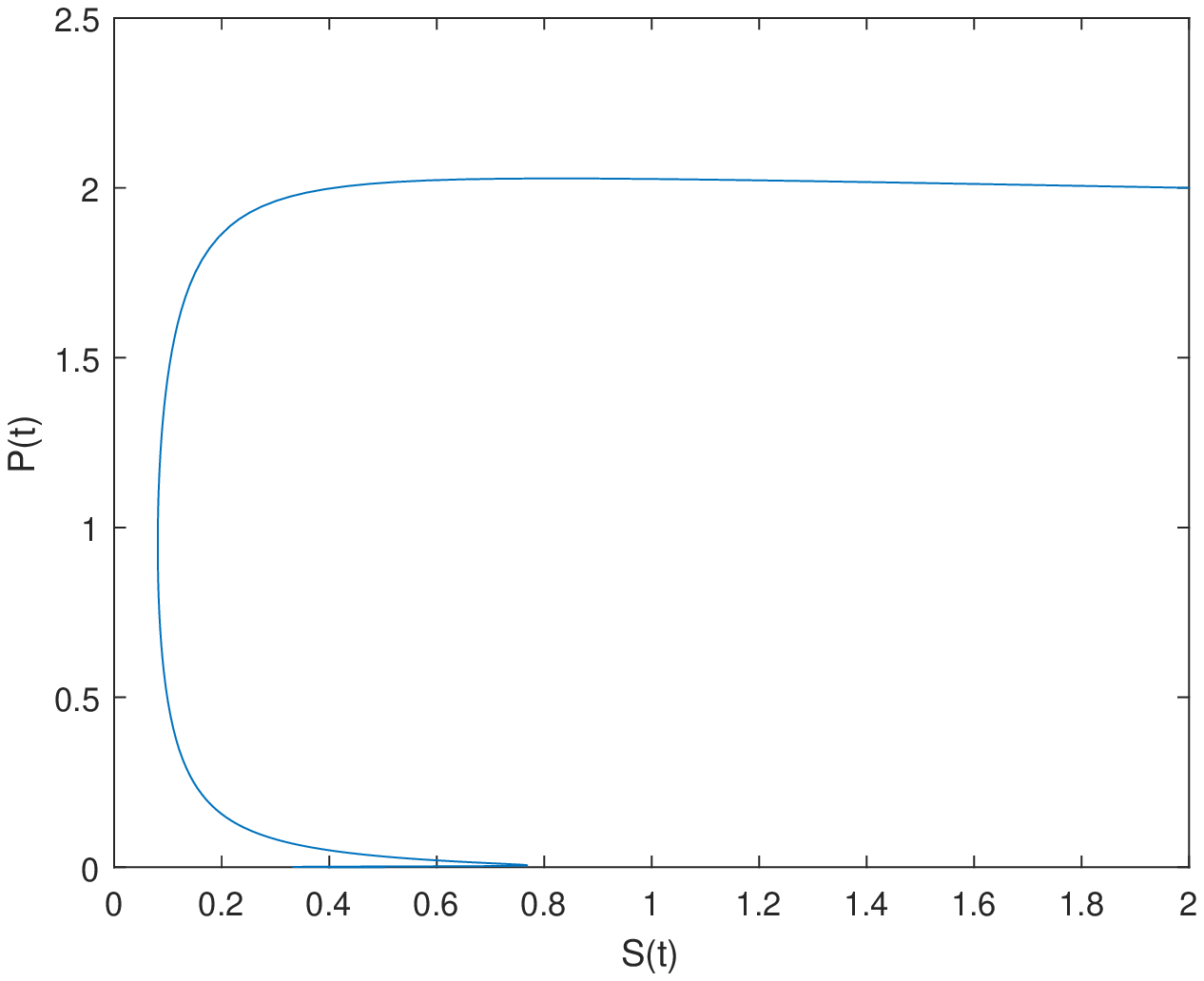}} 
\subfloat{\includegraphics[scale=0.35]{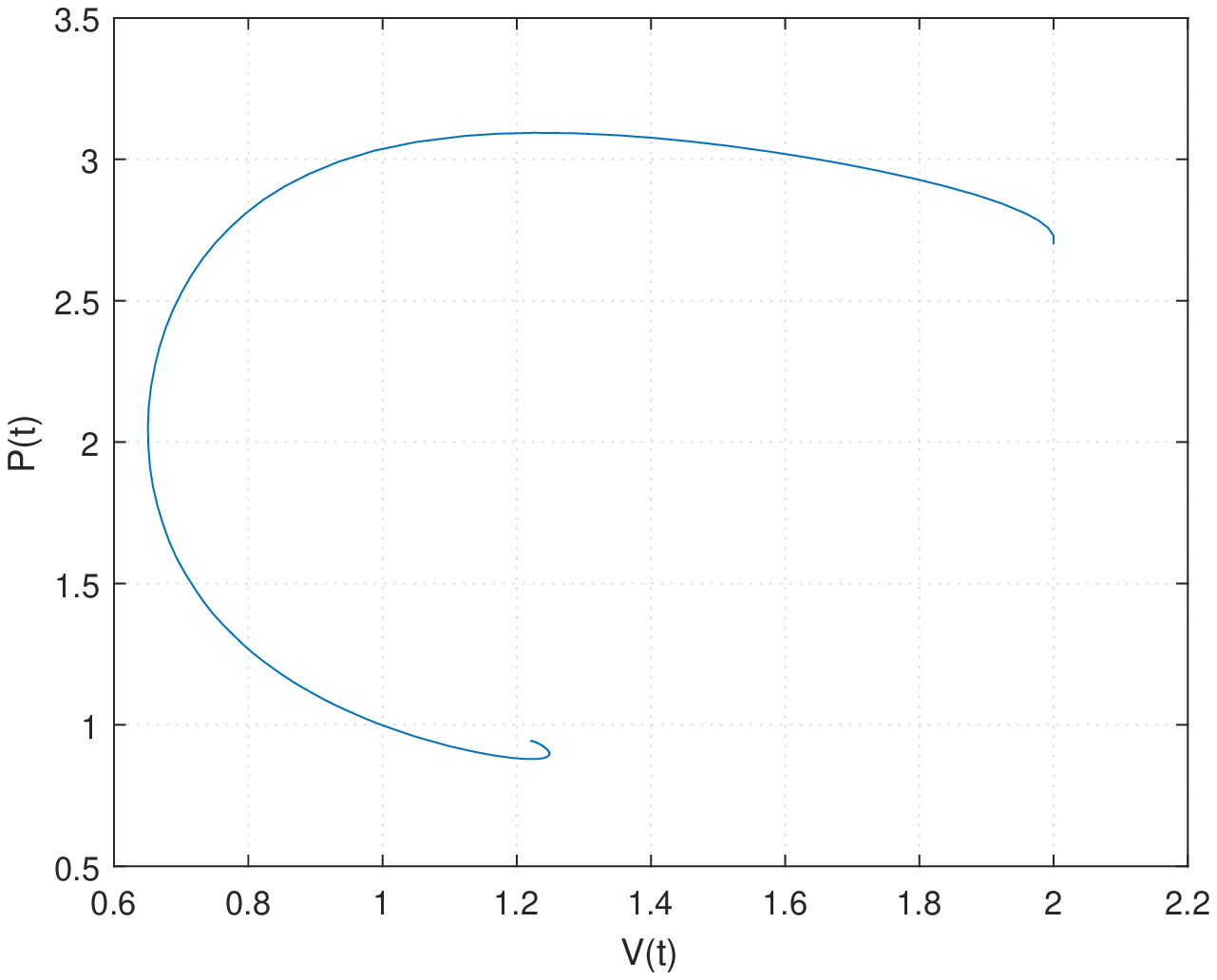}}
\hspace{0mm}
\subfloat{\includegraphics[scale=0.35]{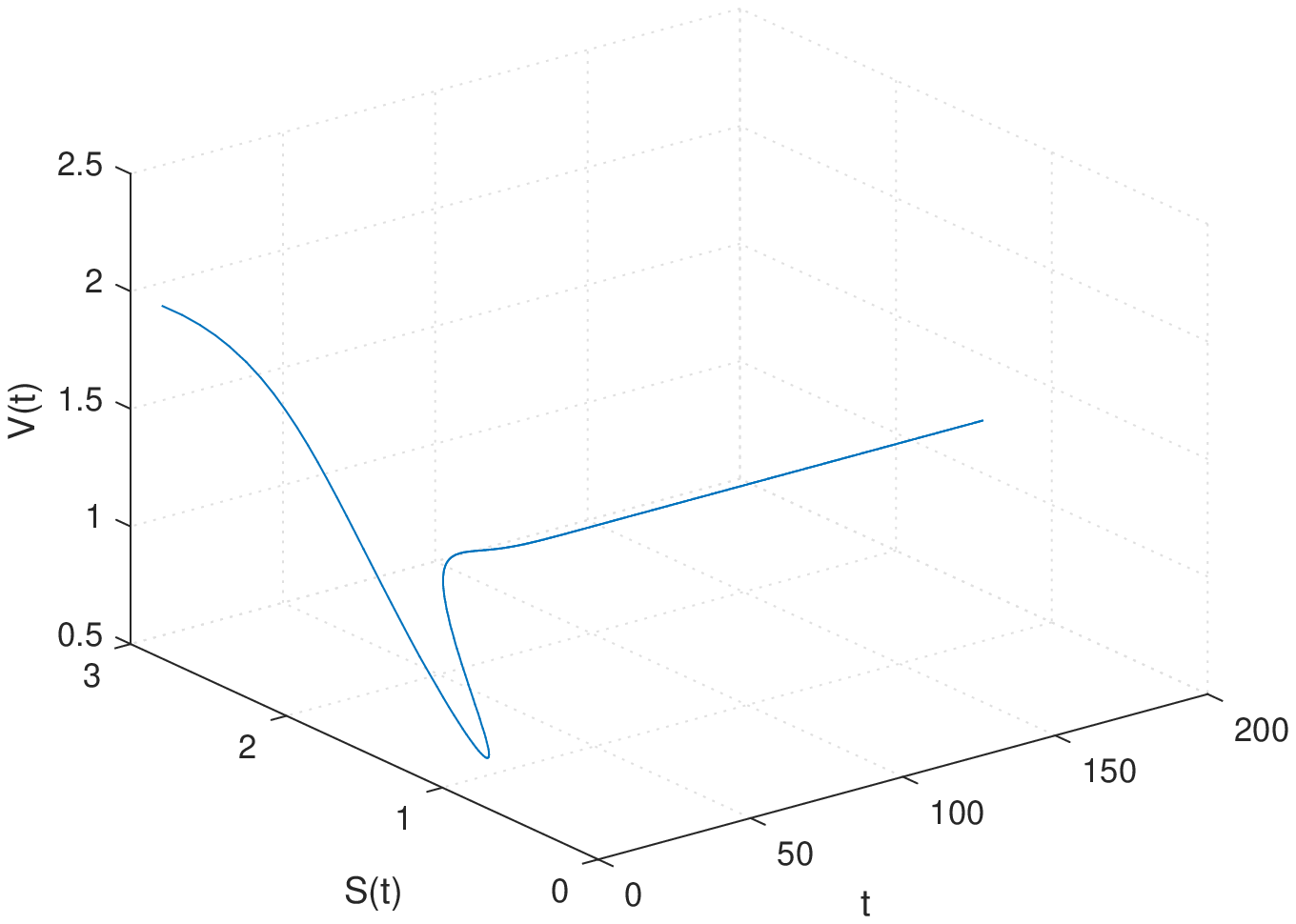}}
\subfloat{\includegraphics[scale=0.35]{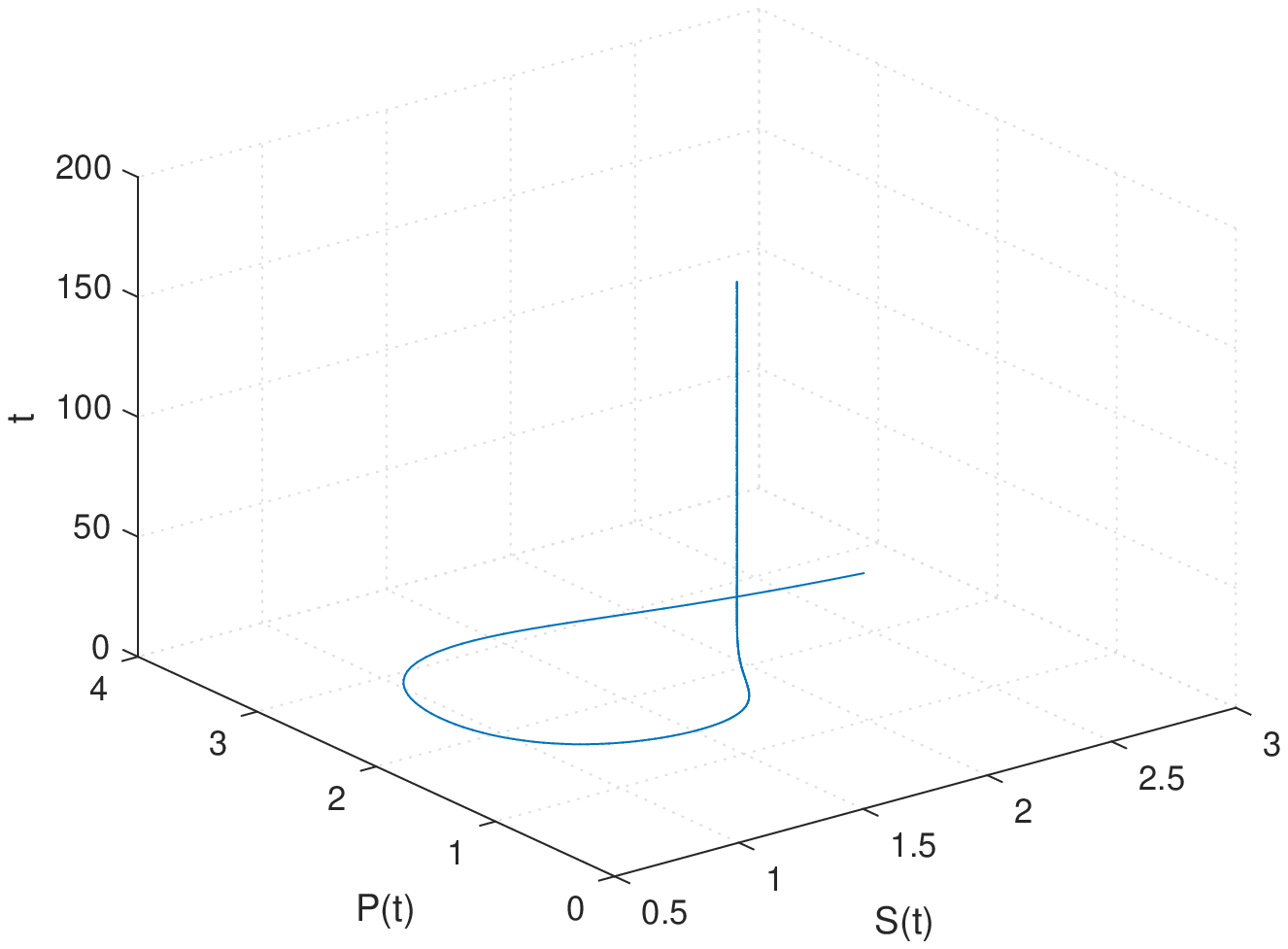}}
\subfloat{\includegraphics[scale=0.35]{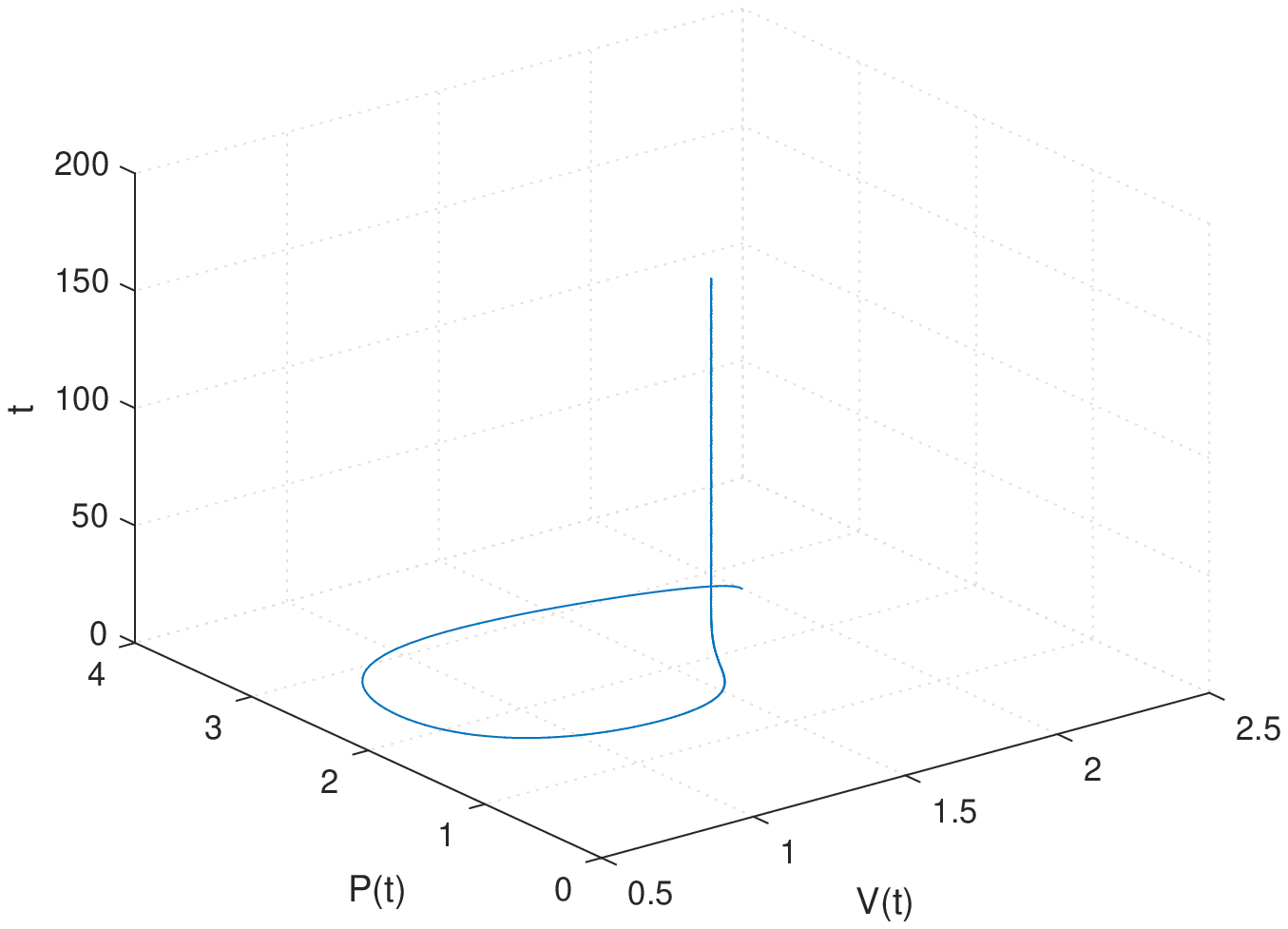}}
\caption{Numerical solutions for the model without disease and migration.}
\label{fig:3}
\end{figure}
\begin{figure}[h!]
\includegraphics[scale=0.5]{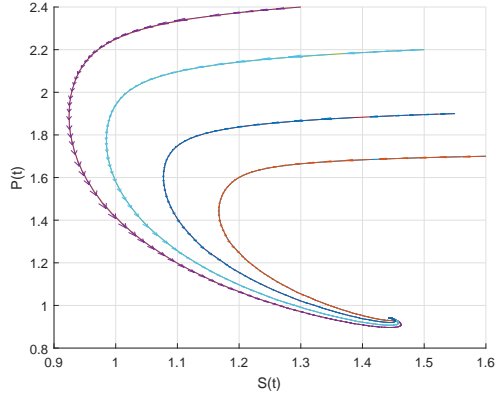}
\caption{Phase portrait showing stability of Interior equilibrium $E^{(2)}$.}
\label{fig:4}
\end{figure}

Now for the set of parameters $P_1$ with $p_1=0.1$ and $p_3=0.1 $, we will analyze our main model by considering two cases. In the first case, we discuss model when there is no migration in prey population. The presence of migration will be discussed in another case with migrating rates in the prey population.
\begin{enumerate}
\item[Case (i)] In the absence of migration, model \eqref{eq1} takes the form:
\begin{align} \label{eq32}
\frac{dS}{dt}&=(1.1) S\left(1-\frac{S+I}{2.9}\right)-(1.2) S I-(1.2) S +(1.2) V-(0.1) P S -(0.25) S, \nonumber \\
\frac{dI}{dt}&=(1.2) S I + (0.2) V I- (0.125) P I- (0.125) I - (0.35)I, \nonumber\\
\frac{dV}{dt}&=(1.2) S - (1.2) V-(0.2) V I- (0.1) P V -(0.1) V,\\
\frac{dP}{dt}&= (0.075) P S + (0.1) P I +(0.075) P V -(0.25) P. \nonumber
\end{align}
\begin{figure}[h!]
\subfloat{\includegraphics[scale=0.44]{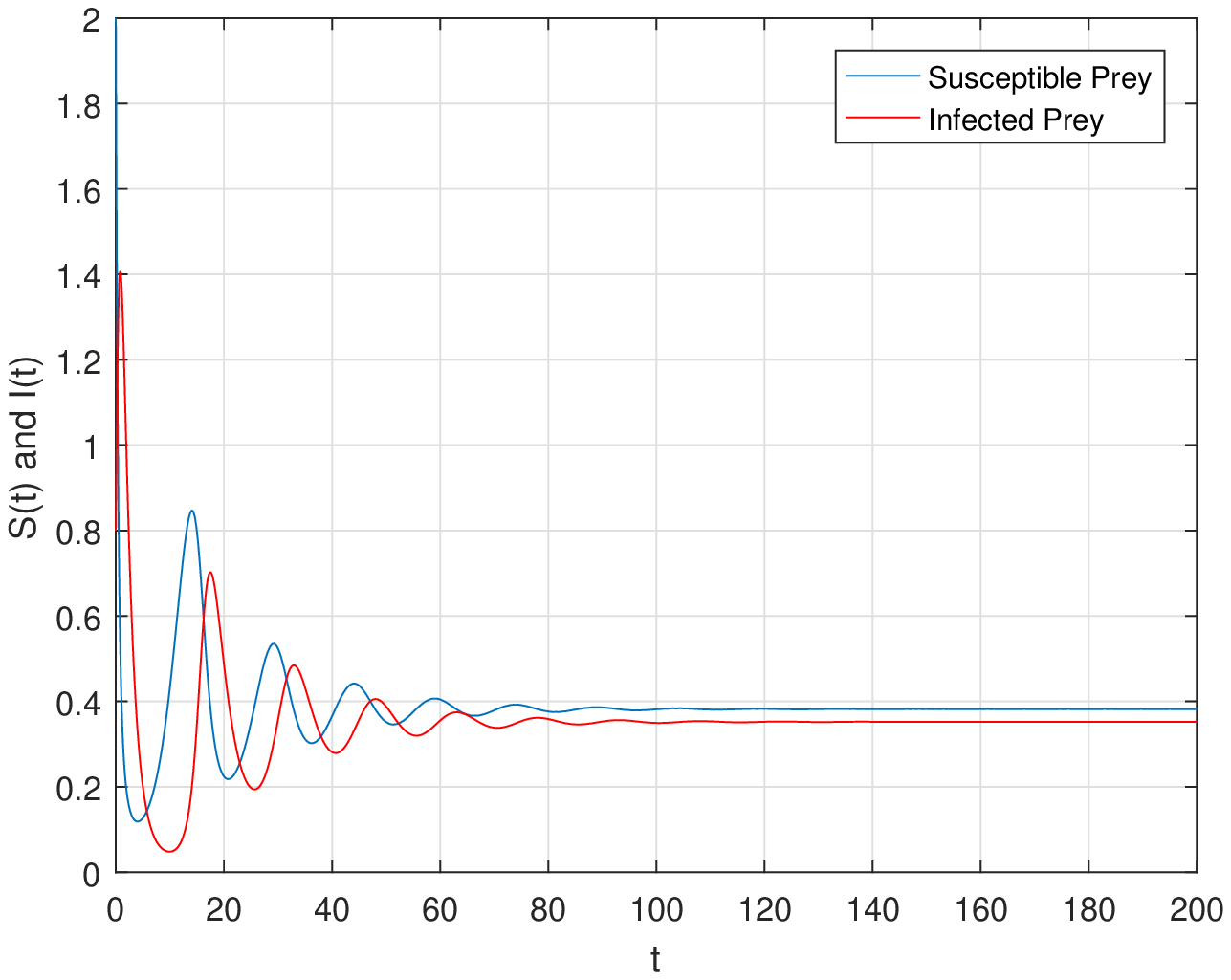}}
\subfloat{\includegraphics[scale=0.44]{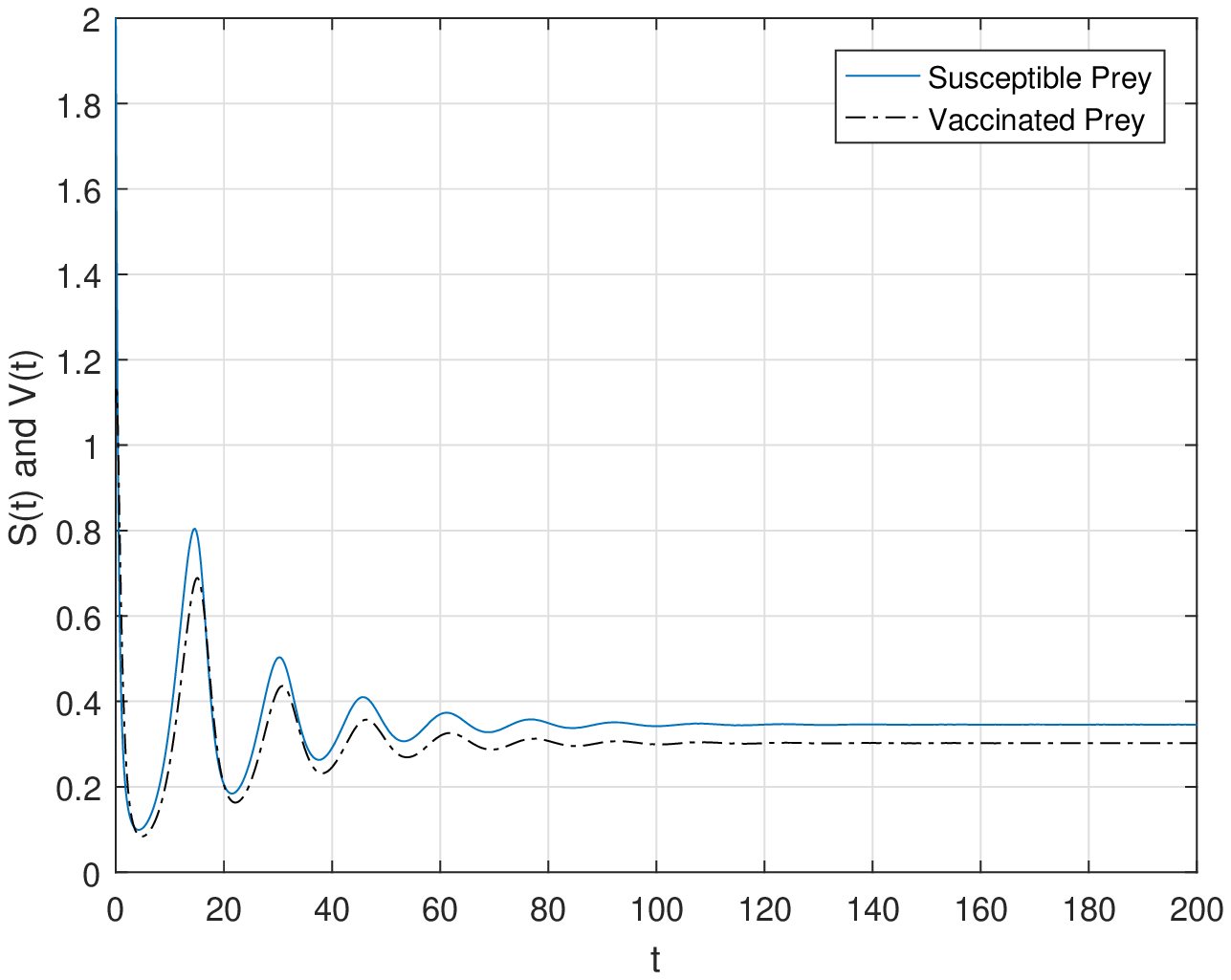}} 
\hspace{0mm}
\subfloat{\includegraphics[scale=0.44]{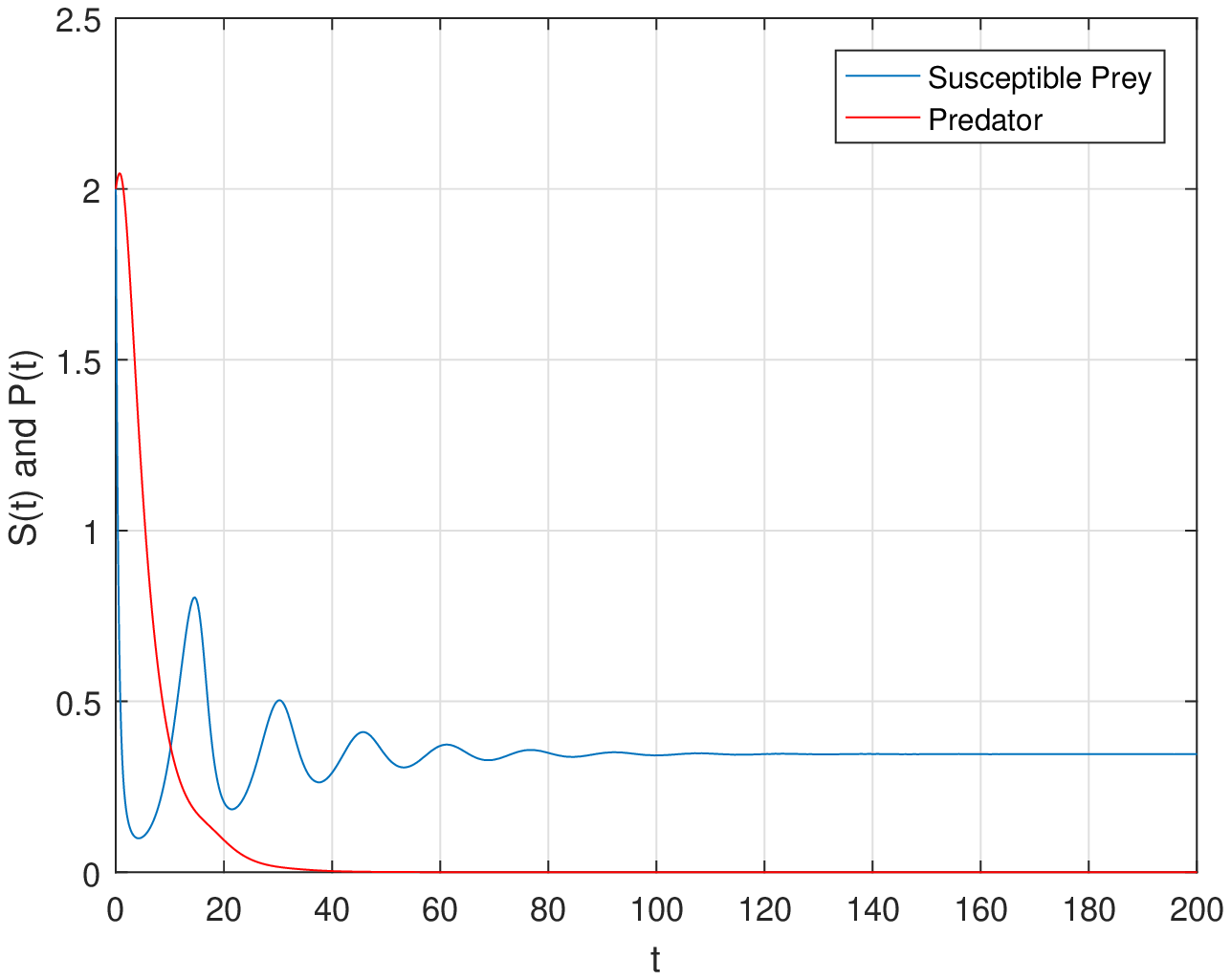}}
\caption{Behavior of solutions for the system \eqref{eq32}.}
\label{fig:5}
\end{figure}
\begin{figure}[t]
\subfloat{\includegraphics[scale=0.35]{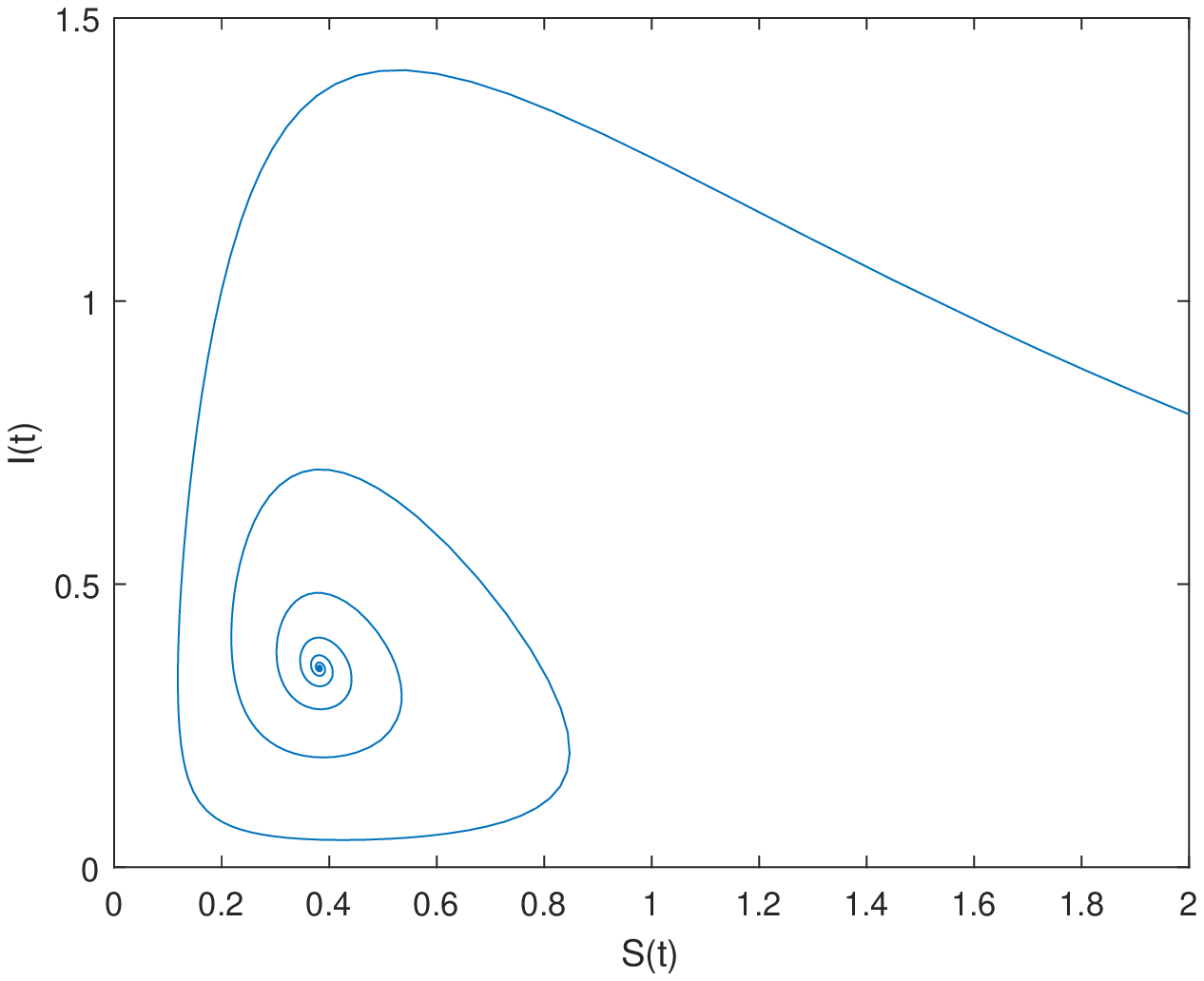}}
\subfloat{\includegraphics[scale=0.35]{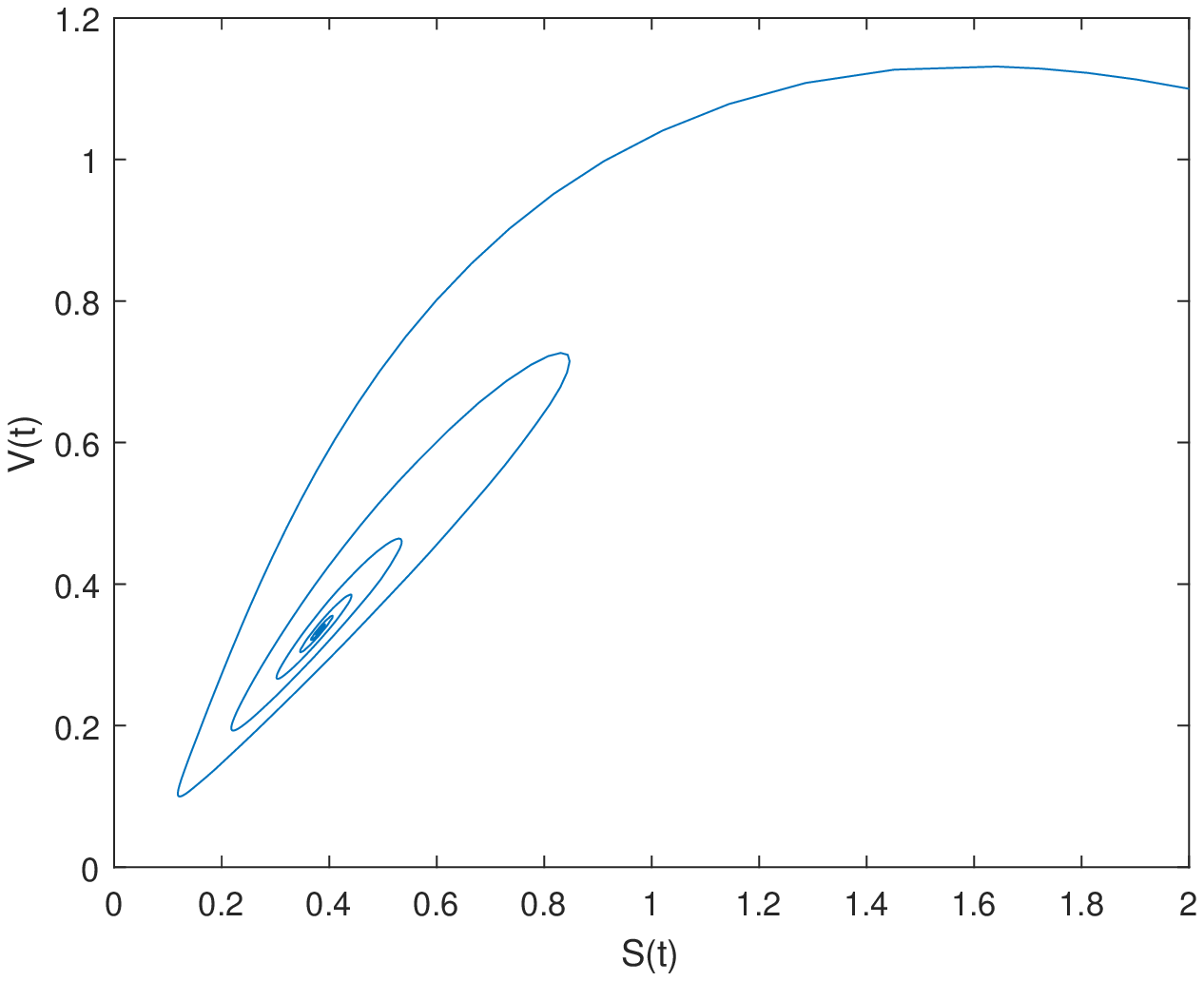}} 
\subfloat{\includegraphics[scale=0.35]{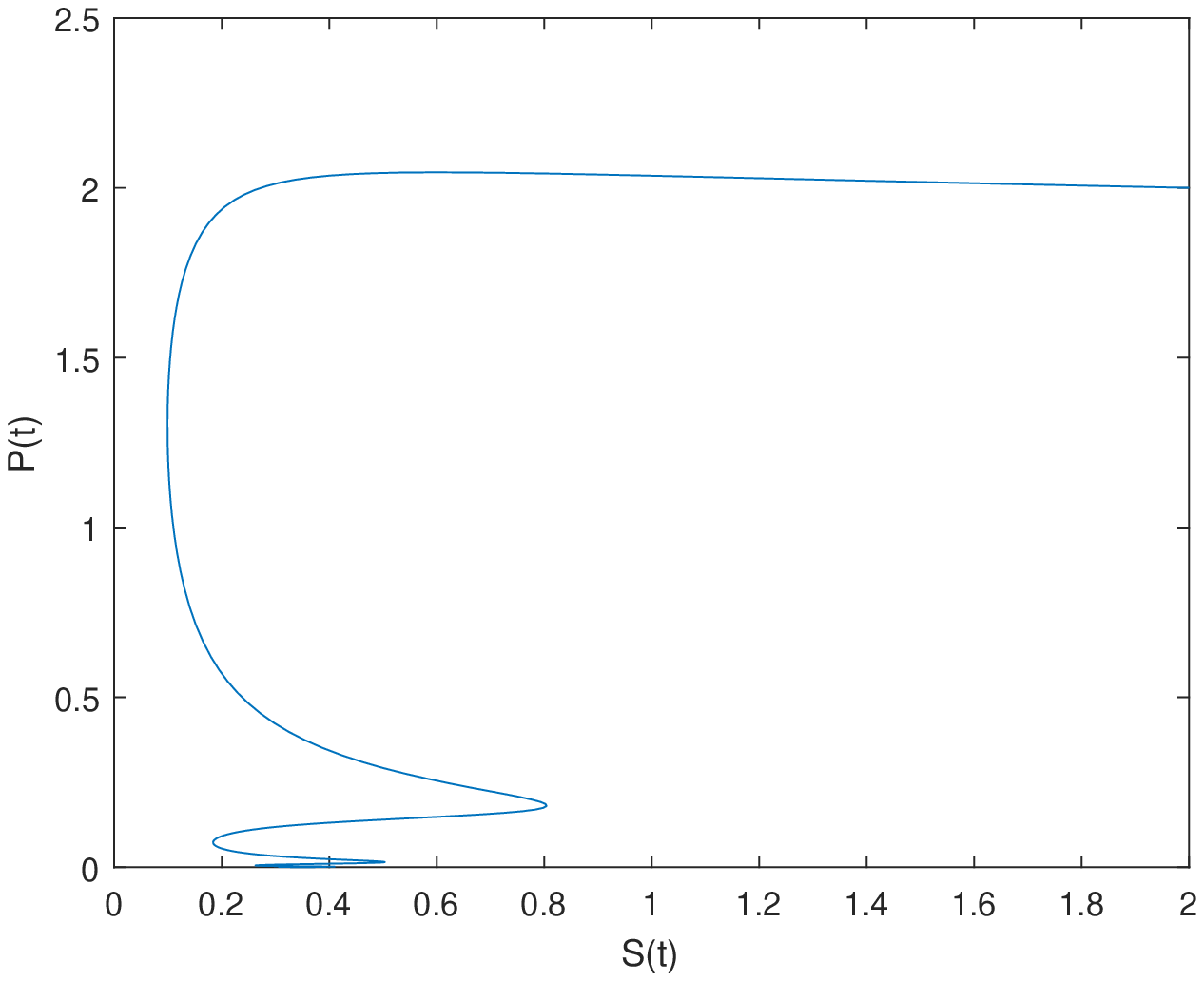}}
\hspace{0mm}
\subfloat{\includegraphics[scale=0.35]{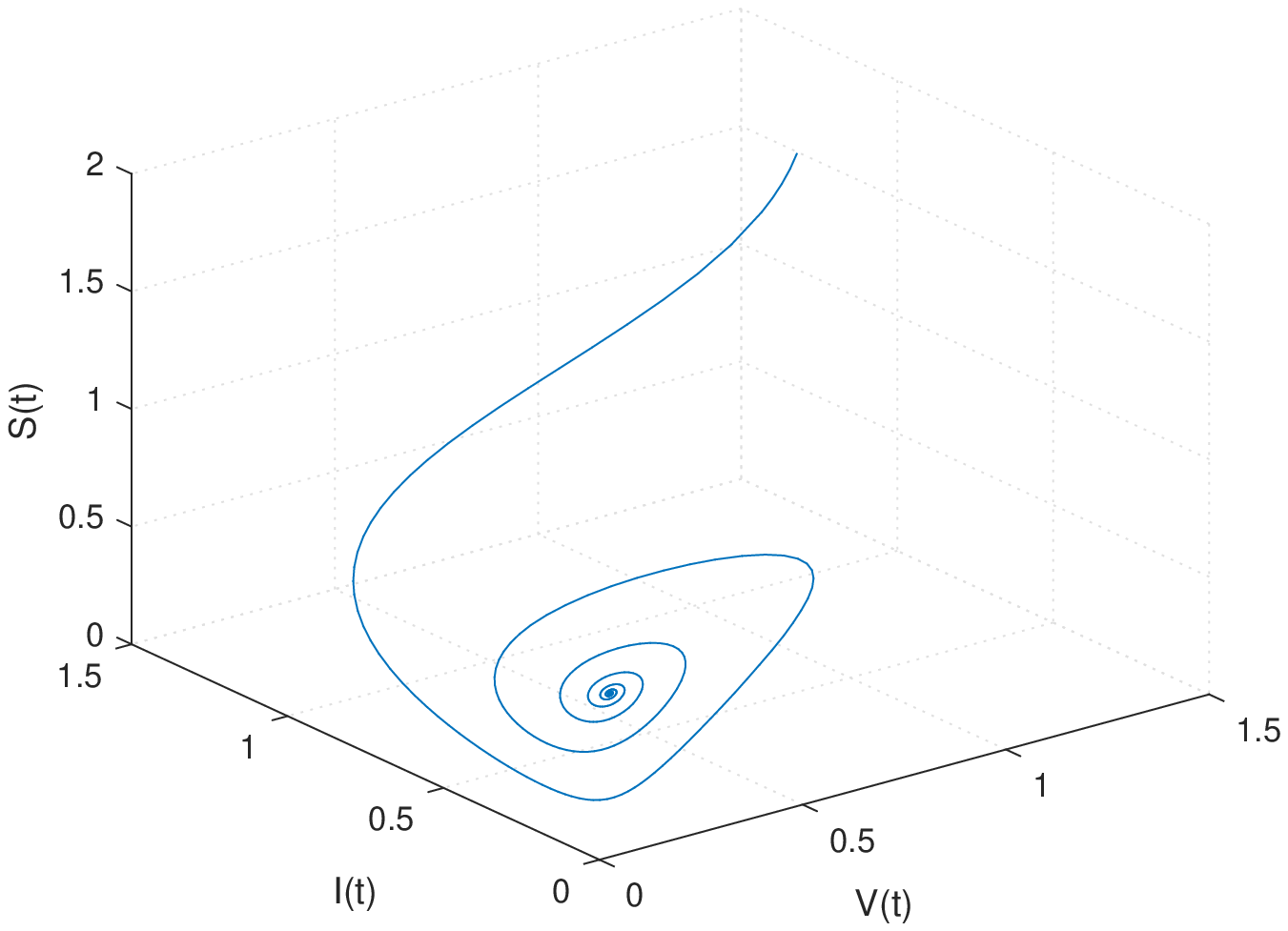}}
\subfloat{\includegraphics[scale=0.35]{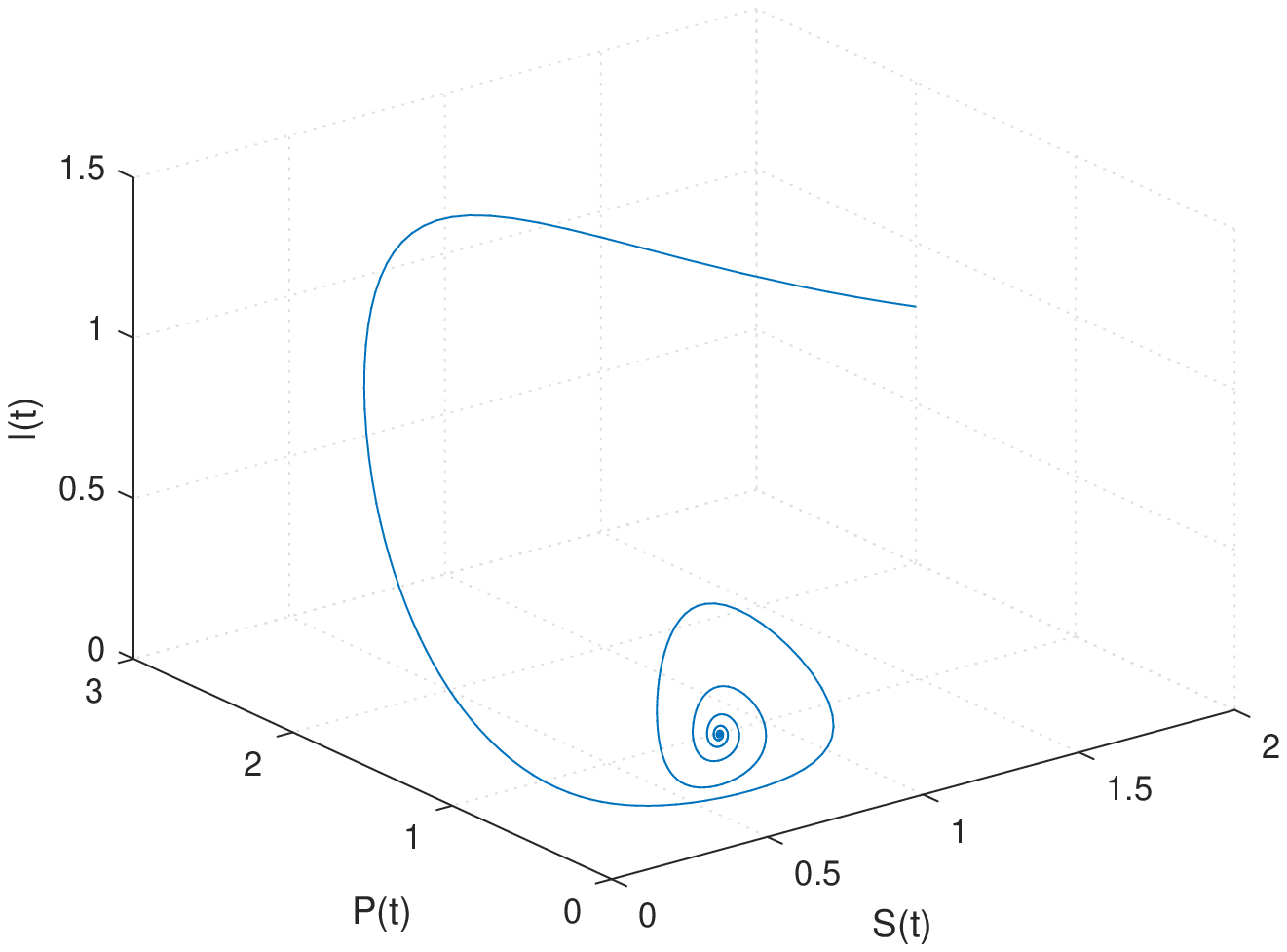}}
\subfloat{\includegraphics[scale=0.35]{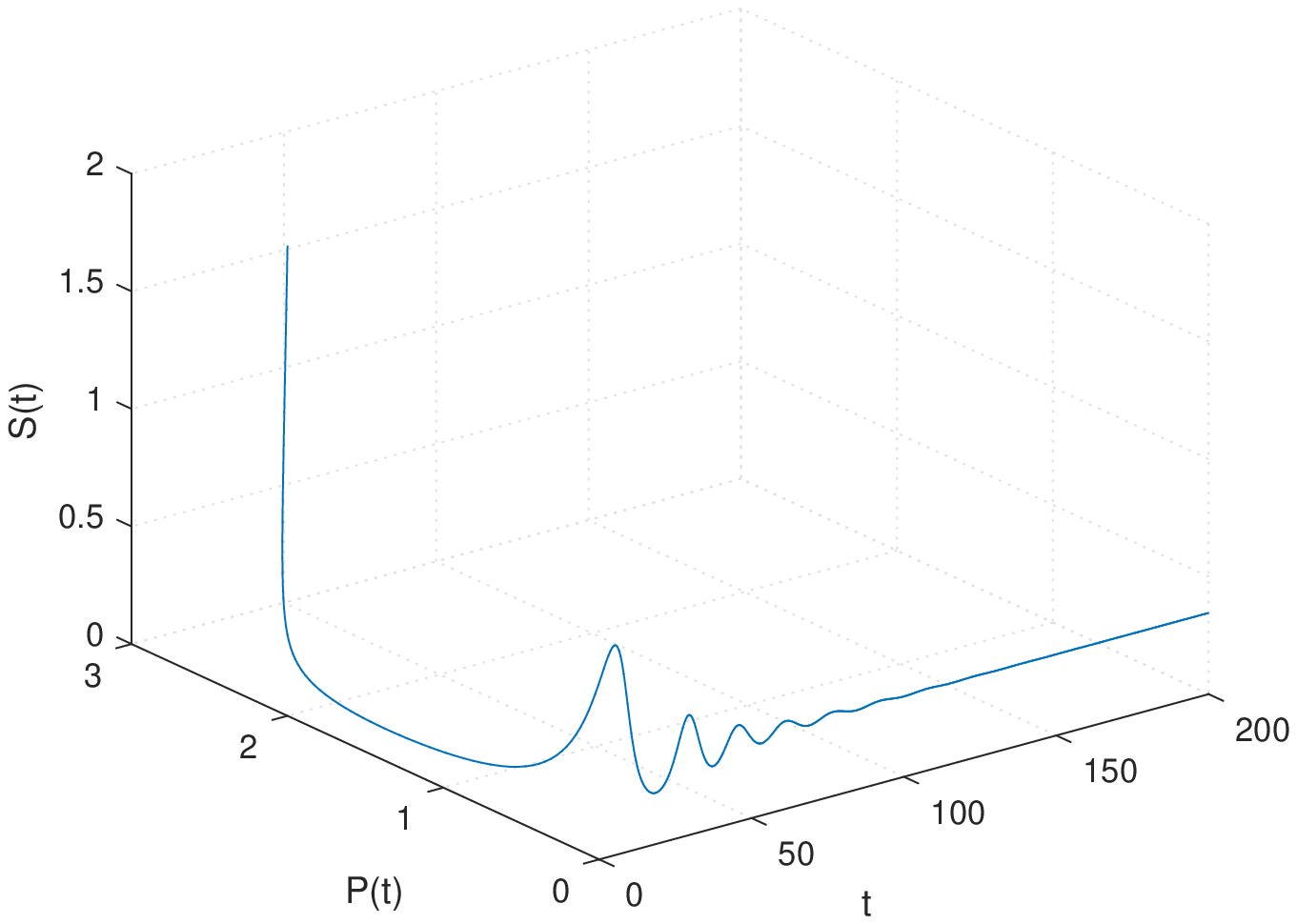}}
\hspace{0mm}
\subfloat{\includegraphics[scale=0.35]{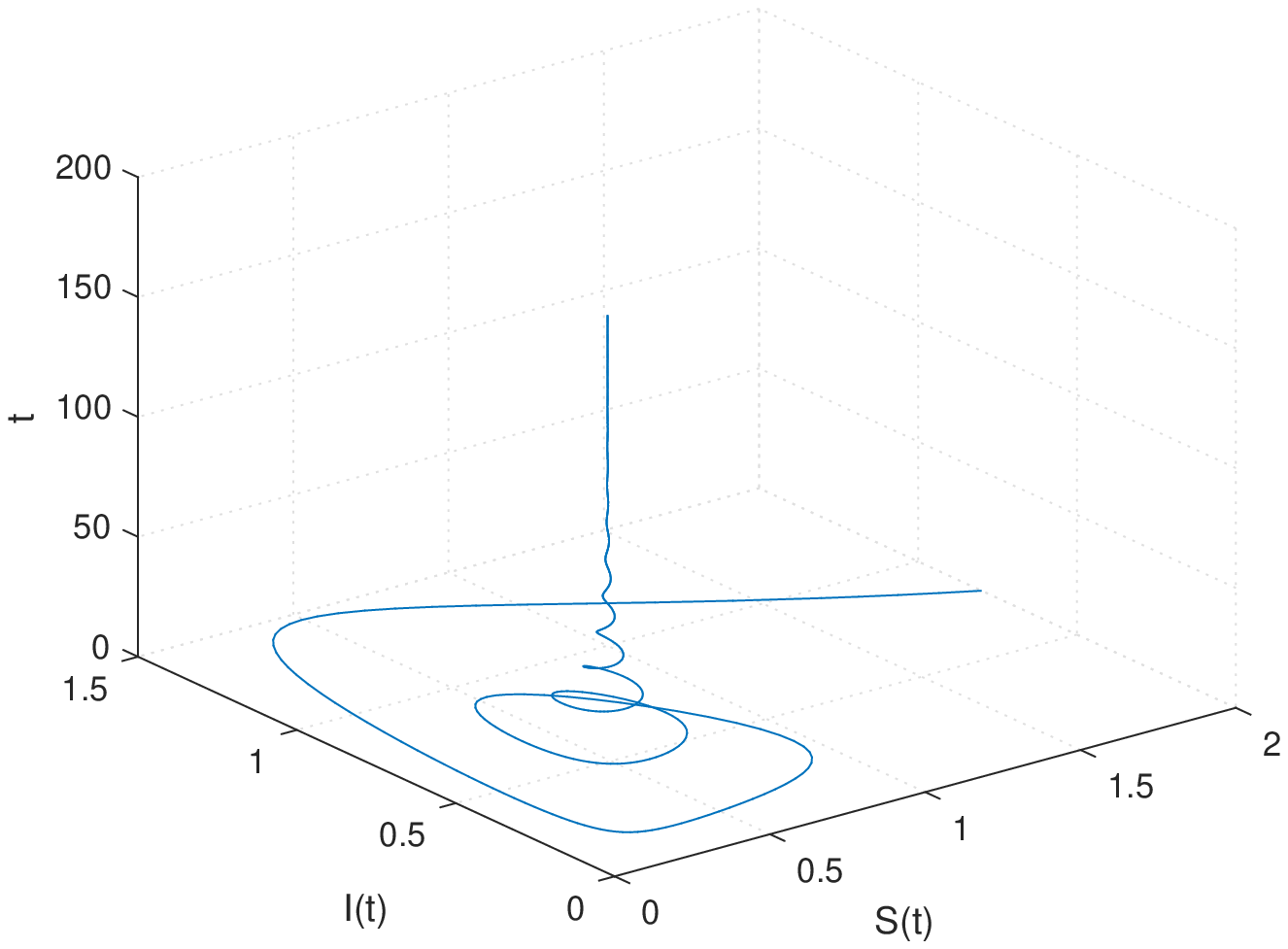}}
\subfloat{\includegraphics[scale=0.35]{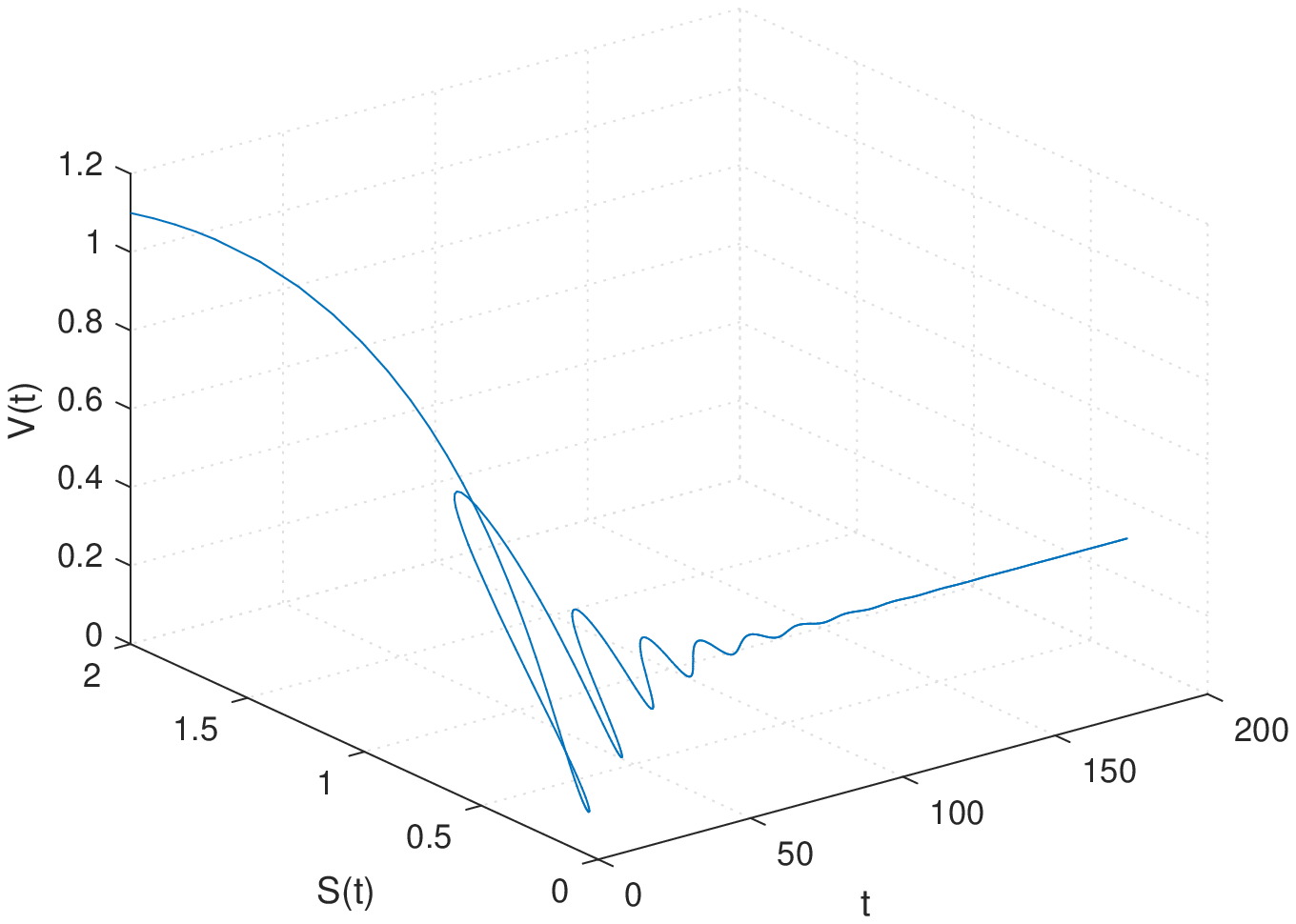}}
\subfloat{\includegraphics[scale=0.35]{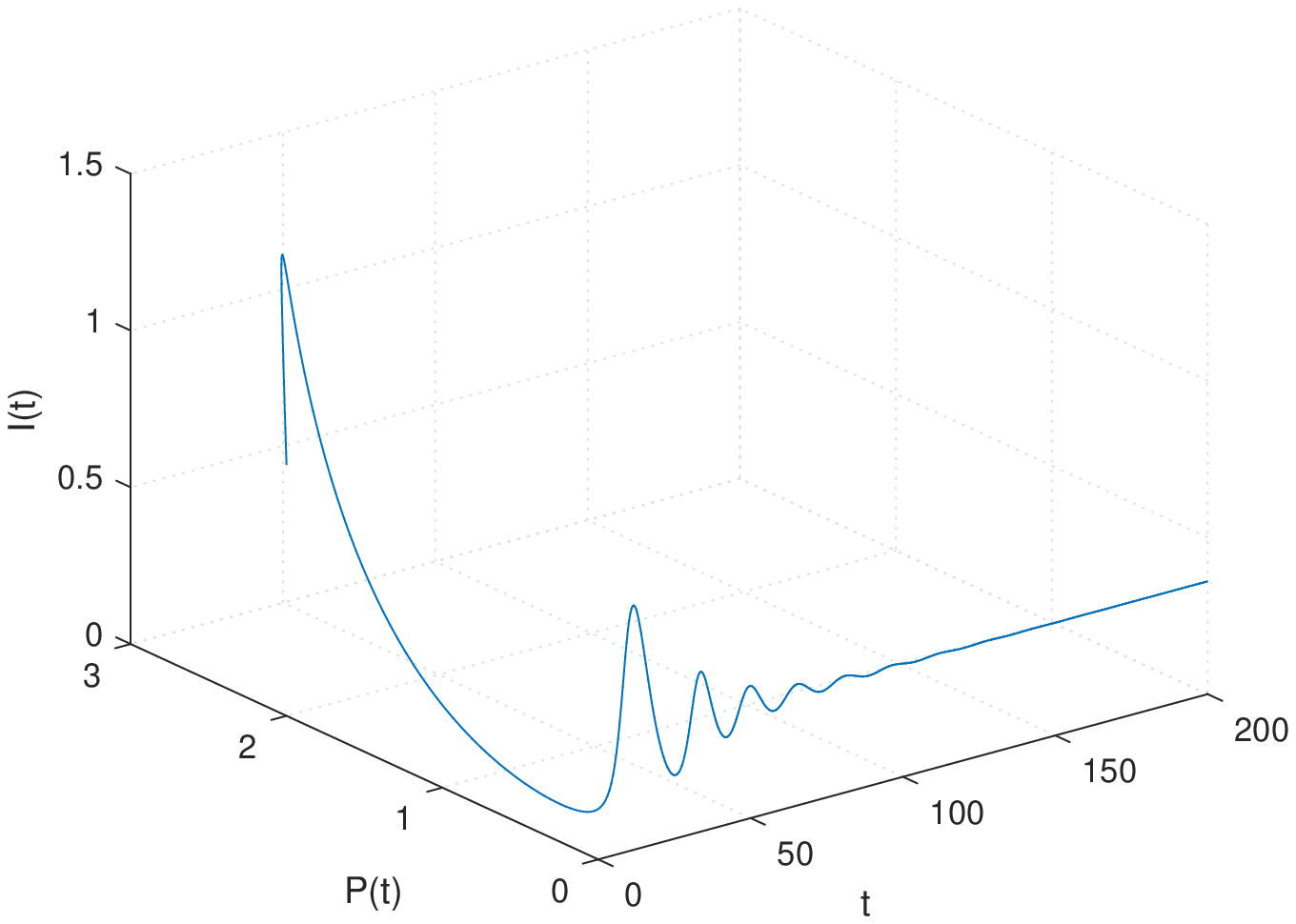}}
\caption{Numerical solutions of the main model without migration.}
\label{fig:6}
\end{figure}

For the system \eqref{eq32}, the equilibrium point $E_0=(0,0,0,0)$ exists but it is not stable as one of the conditions \ref{c3} is not satisfied, that is, $
A=(-r + \theta+ \phi + d_1 + d_3 + m_1 + m_3) = 1.65>0$ but $B=(-r \theta + \theta d_1 - r d_3 + \phi d_3+ d_1 d_3 + \theta m_1 + d_3 m_1 - r m_3 + \phi m_3 +d_1 m_3 + m_1 m_3)=-0.985$, which is not positive. Hence, equilibrium $(E_0)$ is unstable. The equilibrium point $E_1=(1.99755, 0, 1.84389,0)$ exists and one eigenvalue from \eqref{c4} is $\lambda_1 = (-c - d_2- m_2 +\beta S_1+  \sigma V_1) = 2.29084>0$. Therefore, $(E_1)$ is not stable. Equilibrium $E_2=(1.76388, 0, 1.56945, 0.48664)$ exists and one eigenvalue $(-c  - d_2 - m_2 + \beta S_2 +  \sigma V_2 -p_2 P_2) =1.89472 >0$. Thus, equilibrium $(E_2)$ is unstable. The equilibrium point $E_4=(0.345473, 0.359982, 0.302164, 0)$ exists and the equation \eqref{eq8} takes the form:
\begin{equation}\label{en2}
\lambda^3+ 2.5526 \lambda^2 + 0.419829 \lambda +0.406969=0.
\end{equation}
We observe about the conditions \eqref{c5} that:
\begin{enumerate}
\item [(1)]One eigenvalue $(q_1 p_1 S_4 +q_2 p_2 I_4 + q_3 p_3 V_4 -d_4)= -0.165429 < 0$.
\item[(2)] All the coefficients of eq. \eqref{en2} are positive.
\item[(3)] $C_1 C_2 > C_3$ implies that $(2.5526)(0.419829) = 1.07166 > 0.406969$.
\end{enumerate}
Therefore, all the conditions are satisfied. Thus, equilibrium $(E_4)$ is stable.\\
Now we will check the existence of the interior equilibrium point $(E_5)$. After simplification, the equation \eqref{e4} for finding the value of $S_5$ is:
\begin{equation}\label{en3}
S_5^3+ 18.4888 S_5^2+ 54.1917 S_5+ 38.4457=0.
\end{equation}
The roots of eq. \eqref{en3} are $-15.06, -2.33611$ and $-1.09277$. Since none of them is positive, equilibrium $(E_5)$ does not exist.\\
Figures \ref{fig:5} and \ref{fig:6} show the results corresponding to the system \eqref{eq32}.

\noindent
The basic reproduction number is estimated as $\mathcal{R}_0= 5.822817>1$. Therefore, disease is endemic in this case.
\newline
\item[Case (ii)] In the presence of migration with migration rates $m_1=0.25, m_2=0.125$ and $m_3=0.25$. Thus, system \eqref{eq1} will become:
\begin{align} \label{eq33}
\frac{dS}{dt}&=(1.1) S\left(1-\frac{S+I}{2.9}\right)-(1.2) S I-(1.2) S +(1.2) V-(0.1) P S -(0.25)S-(0.25) S, \nonumber \\
\frac{dI}{dt}&=(1.2) S I + (0.2) V I- (0.125) P I- (0.125) I-(0.125) I - (0.35)I, \nonumber\\
\frac{dV}{dt}&=(1.2) S - (1.2) V-(0.2) V I- (0.1) P V-(0.25) V -(0.1) V,\\
\frac{dP}{dt}&= (0.075) P S + (0.1) P I +(0.075) P V -(0.25) P. \nonumber
\end{align}
\begin{figure}[h!]
\includegraphics[scale=0.56]{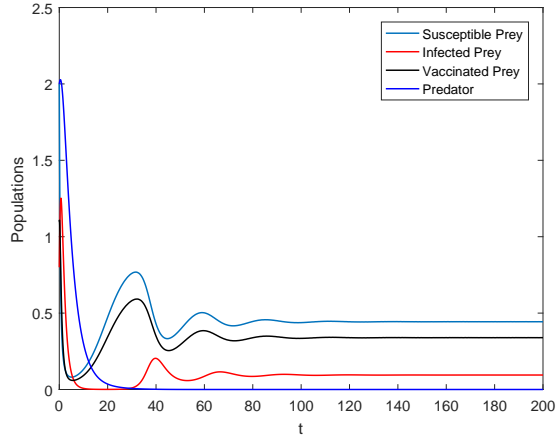}
\caption{Behavior of solutions for the system \eqref{eq33}.}
\label{fig:7}
\end{figure}
\begin{figure}[h!]
\subfloat{\includegraphics[scale=0.35]{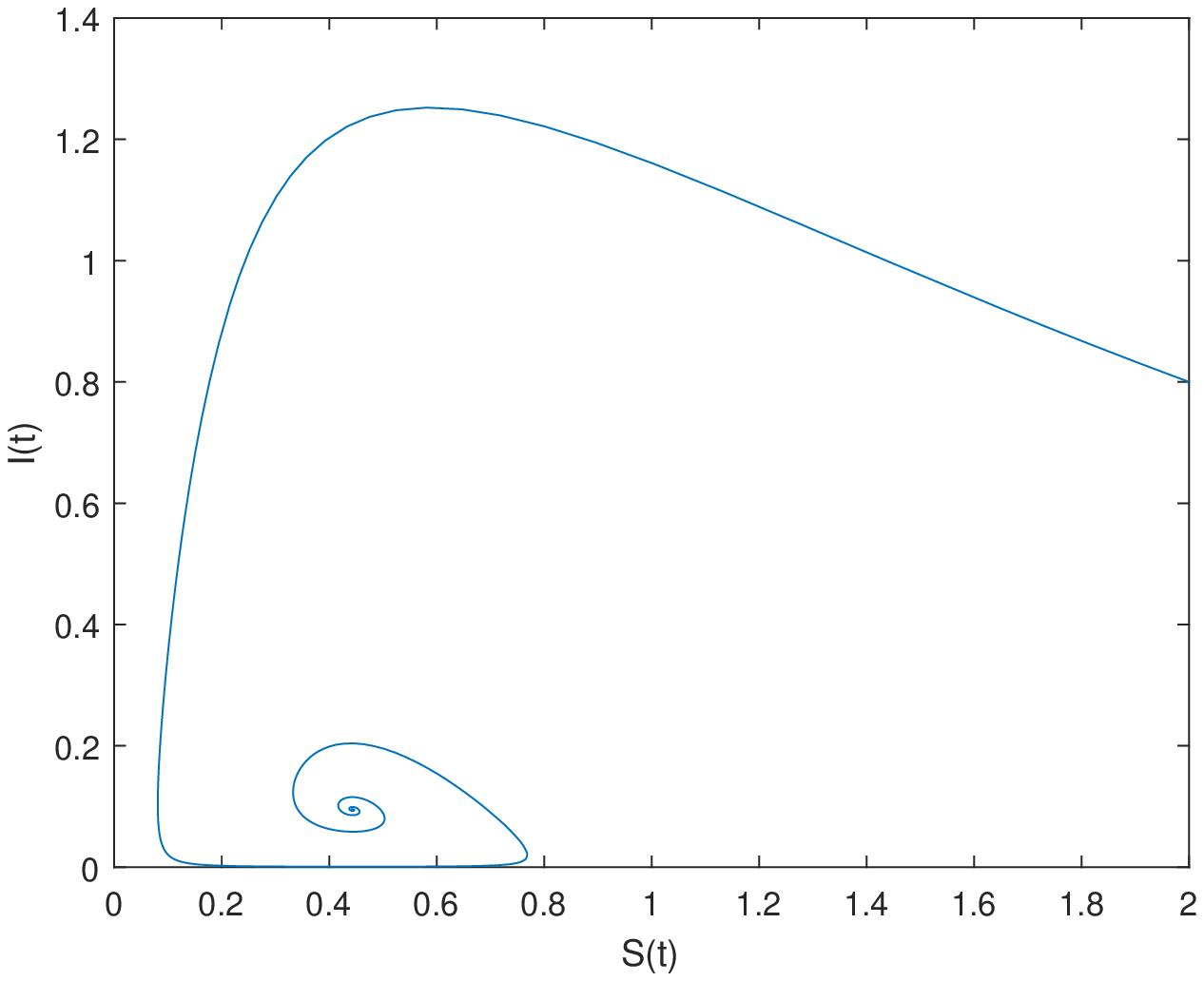}}
\subfloat{\includegraphics[scale=0.35]{SV2}} 
\subfloat{\includegraphics[scale=0.35]{SP2}}
\hspace{0mm}
\subfloat{\includegraphics[scale=0.35]{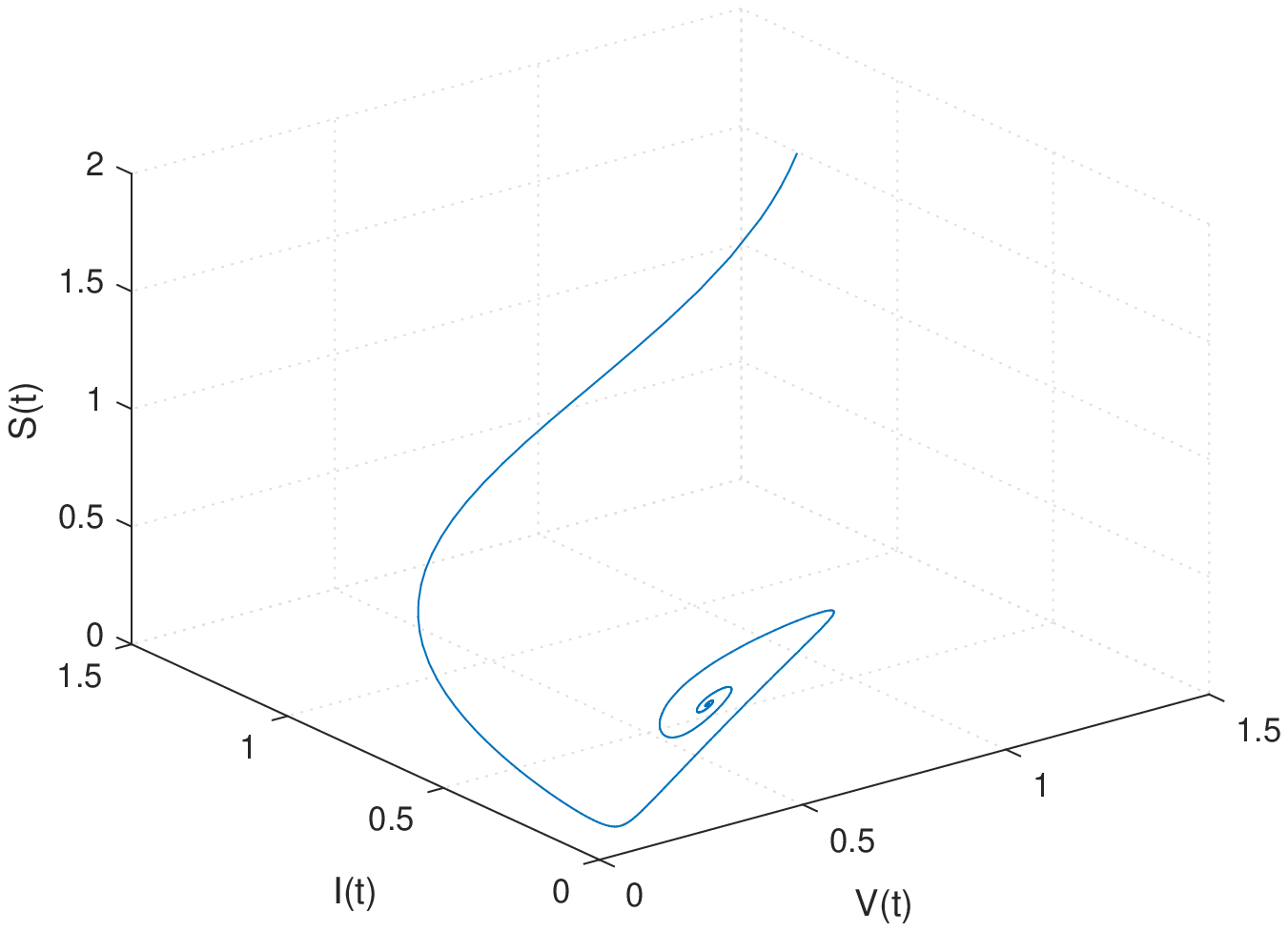}}
\subfloat{\includegraphics[scale=0.35]{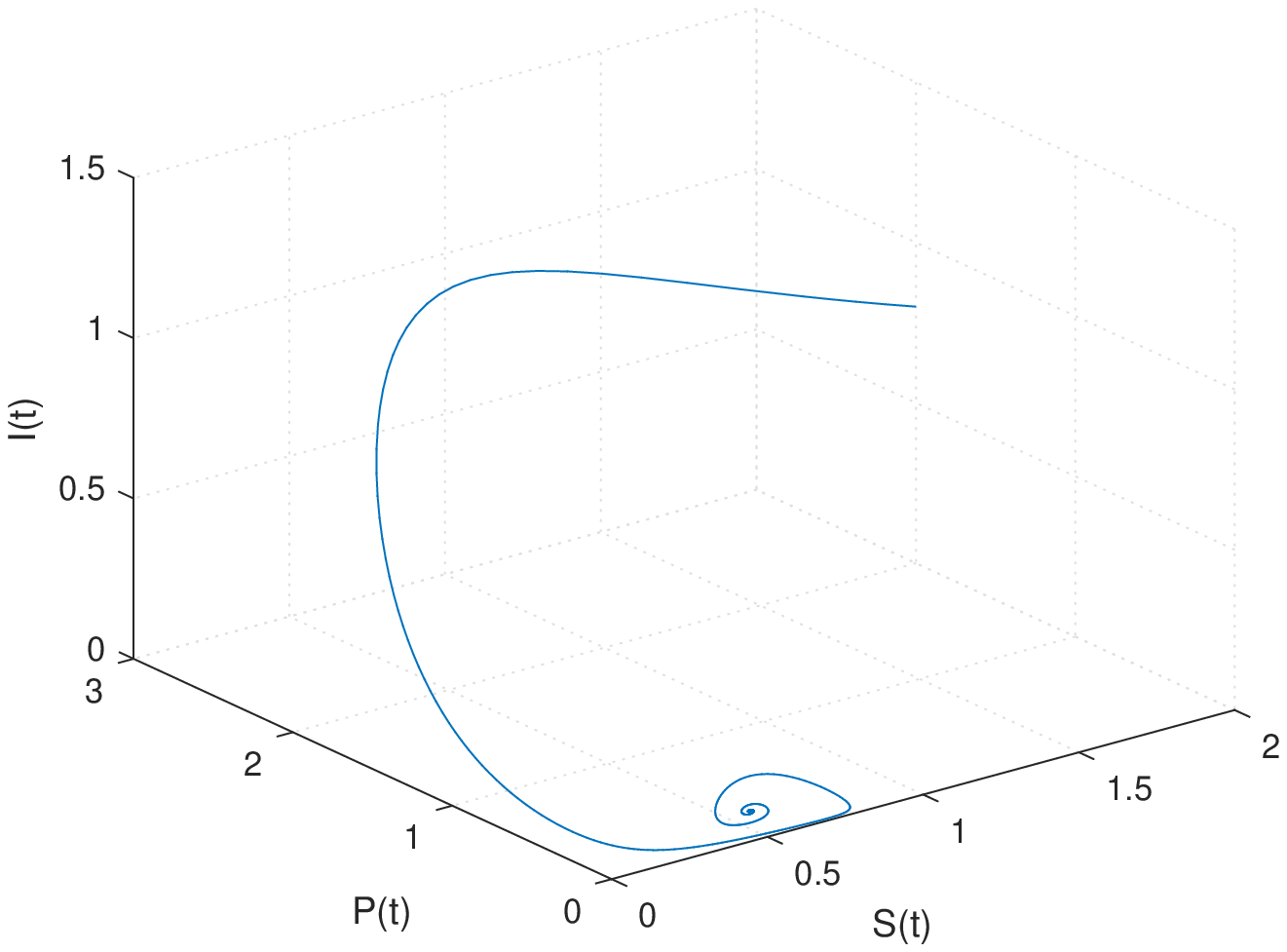}}
\subfloat{\includegraphics[scale=0.35]{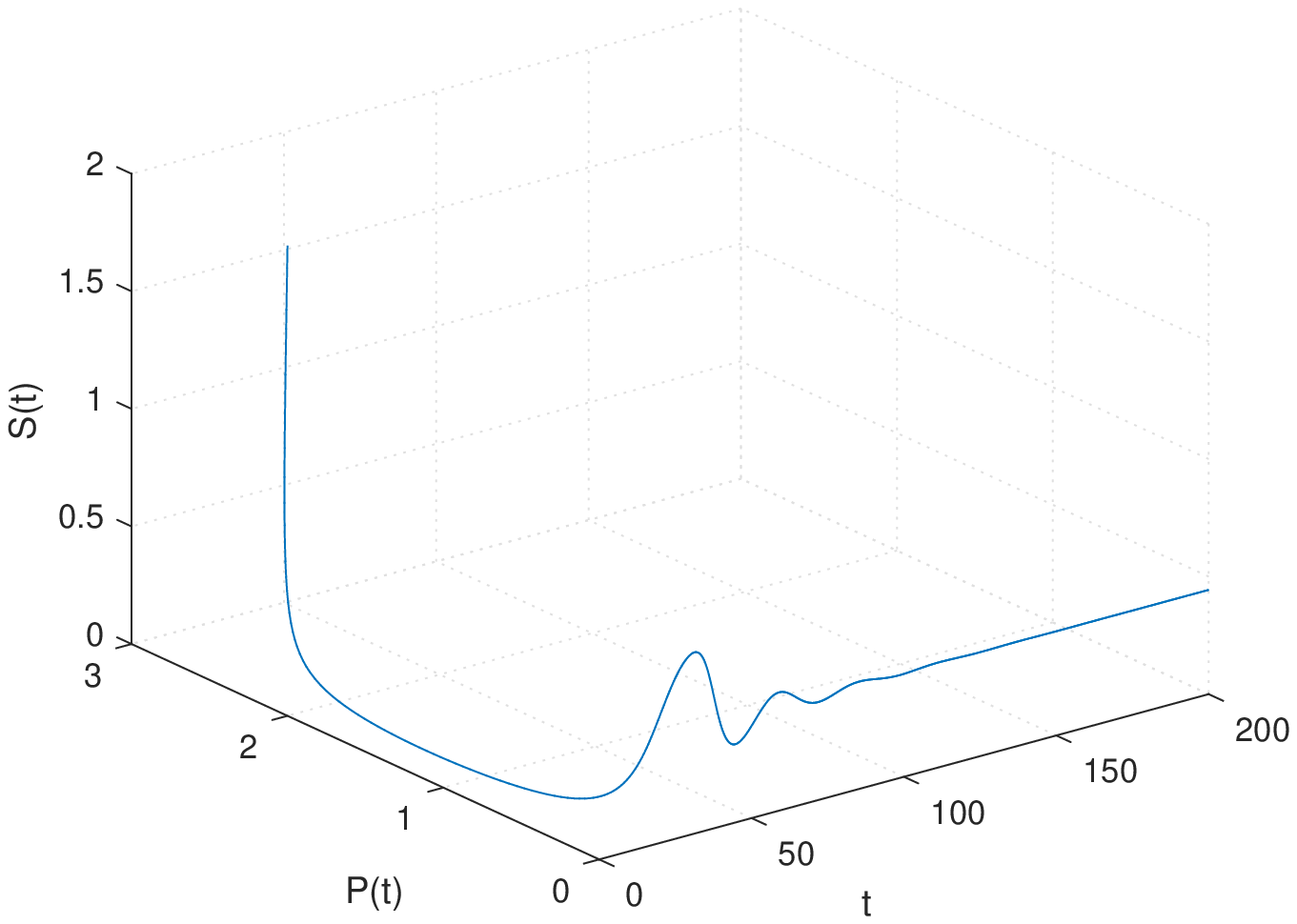}}
\hspace{0mm}
\subfloat{\includegraphics[scale=0.35]{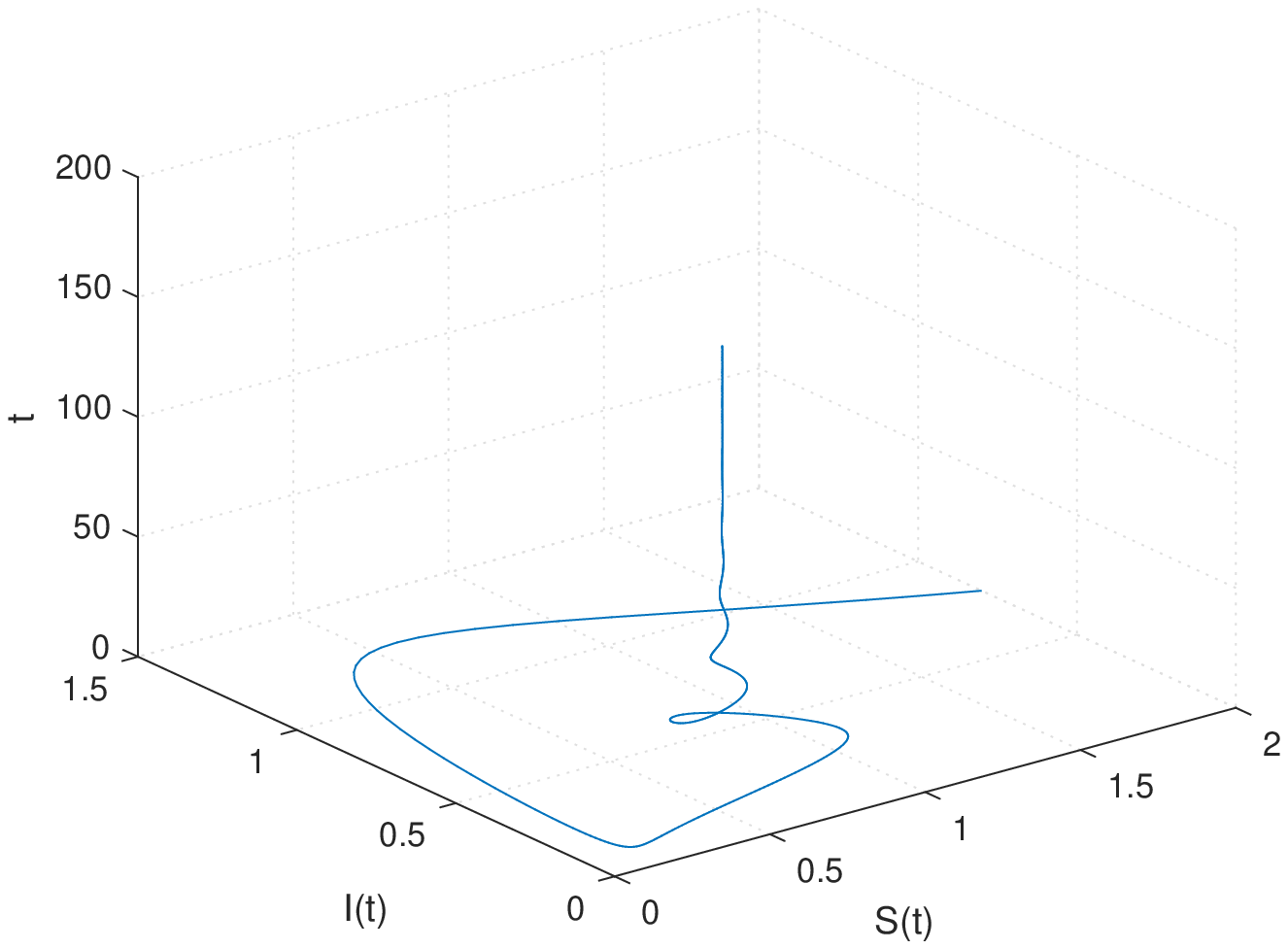}}
\subfloat{\includegraphics[scale=0.35]{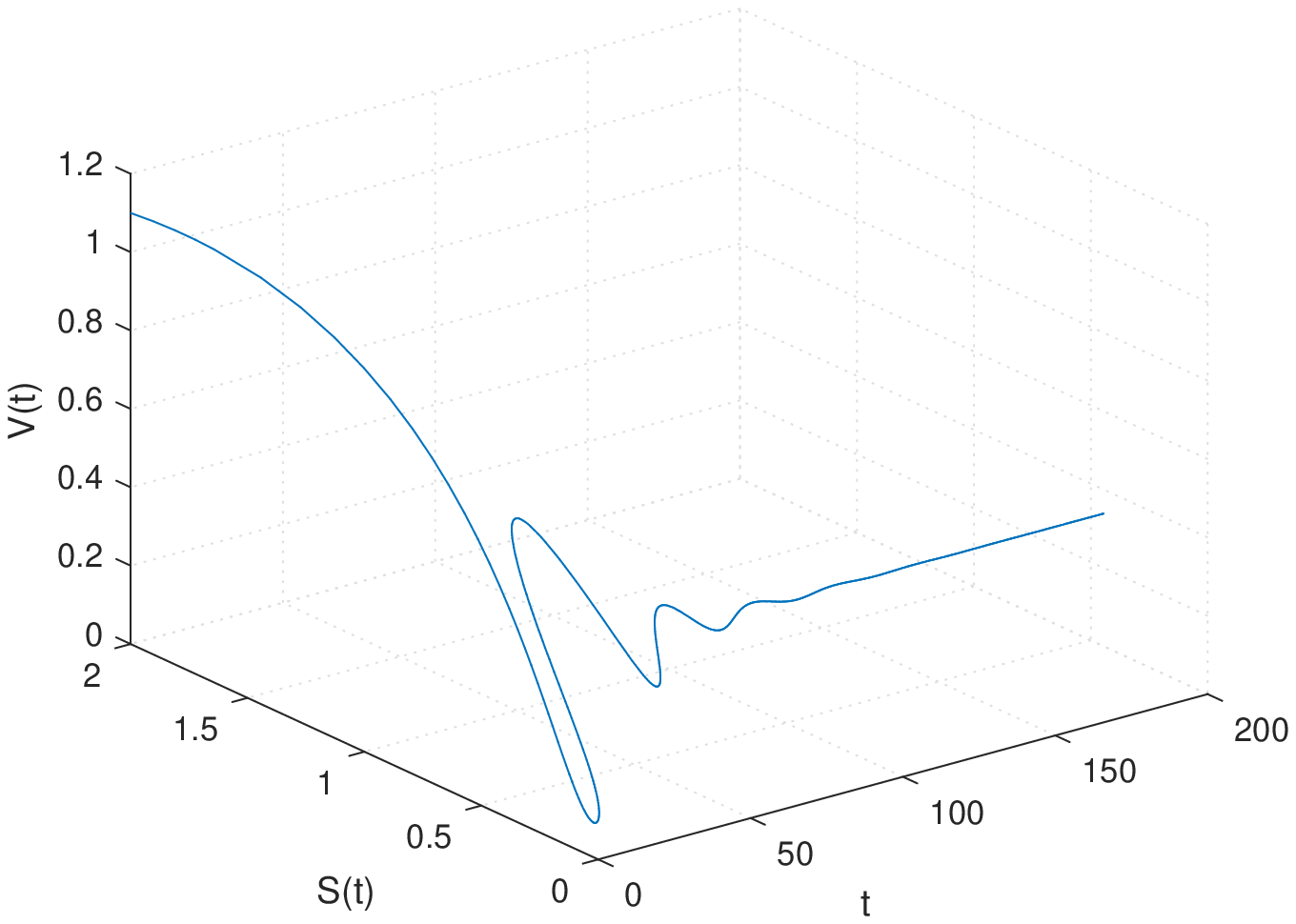}}
\subfloat{\includegraphics[scale=0.35]{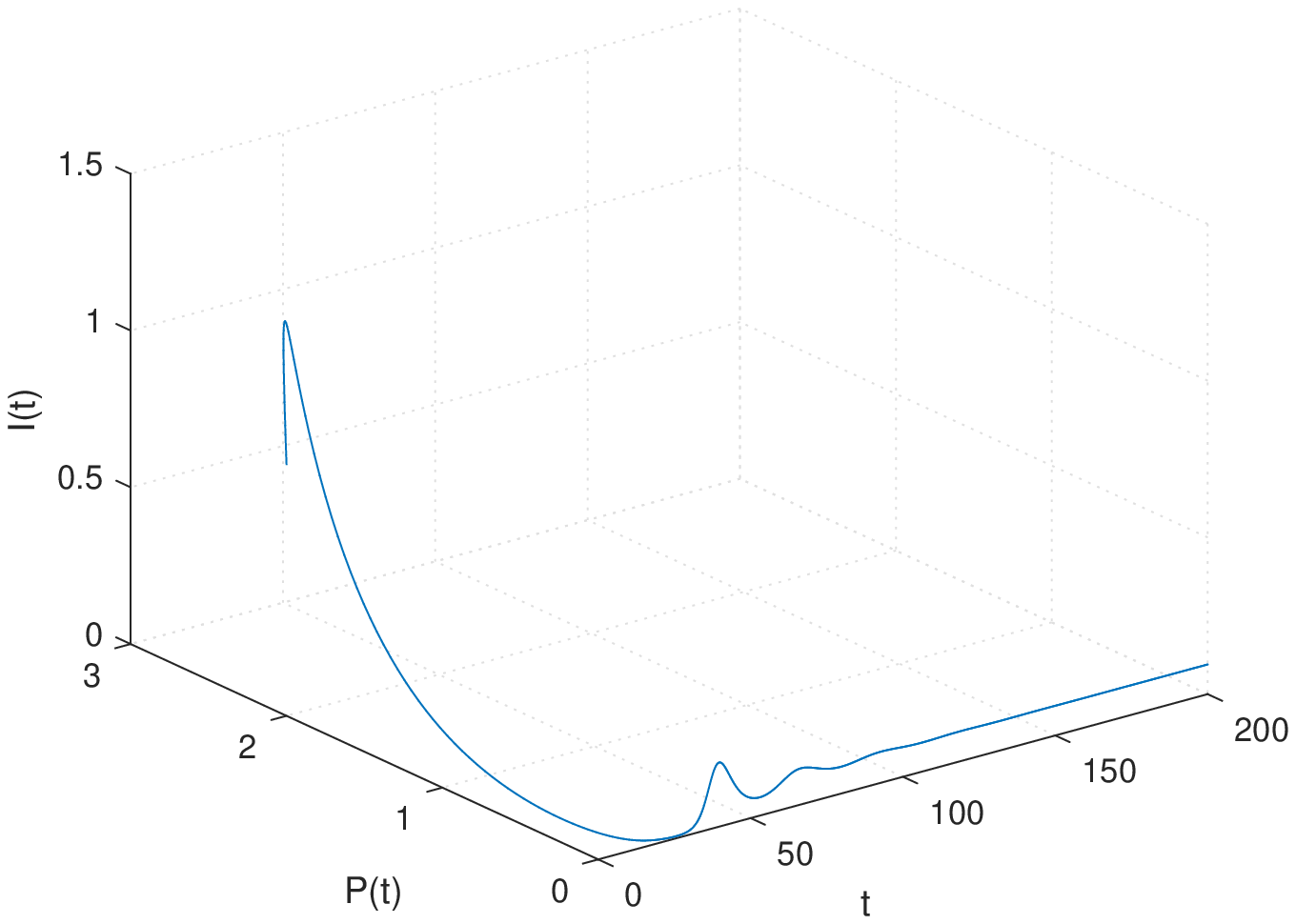}}
\caption{Numerical solutions for the main model with migration.}
\label{fig:8}
\end{figure}

For the system \eqref{eq33}, the equilibrium point $E_0=(0,0,0,0)$ exists but it is not stable as one of the conditions \eqref{c3} is not satisfied, that is, $
A=(-r + \theta+ \phi + d_1 + d_3 + m_1 + m_3) = 2.15>0$ but $B=(-r \theta + \theta d_1 - r d_3 + \phi d_3+ d_1 d_3 + \theta m_1 + d_3 m_1 - r m_3 + \phi m_3 +d_1 m_3 + m_1 m_3)=-0.51<0.$ Hence, equilibrium $(E_0)$ is unstable. The equilibrium point $E_1=(0.867449, 0, 0.671573,0)$ exists but one eigenvalue from \eqref{c4} is $\lambda_1 = (-c - d_2- m_2 +\beta S_1+  \sigma V_1) = 0.575253>0$. Therefore, $(E_1)$ is not stable.\\
\noindent
On simplification, the equation \eqref{eqn2} for finding the value of $P_2$ is:
\begin{equation}\label{en4}
P_2^3+ 61.6437 P_2^2+ 932.204 P_2+ 1635.14=0.
\end{equation}
The roots of eq. \eqref{en4} are $-2.01336, -21.0518$ and $-38.5785$. Equilibrium $(E_2)$ does not exist as none of the roots is positive. \\

\noindent
The equilibrium point $E_4=(0.443469, 0.0947259, 0.339185, 0)$ exists and the characteristic equation \eqref{eq8} is
\begin{equation}\label{en5}
\lambda^3+ 2.65497 \lambda^2 + 0.344814 \lambda +0.151479=0.
\end{equation}
We observe about the conditions \eqref{c5} that:
\begin{enumerate}
\item [(1)]One eigenvalue $(q_1 p_1 S_4 +q_2 p_2 I_4 + q_3 p_3 V_4 -d_4)= -0.181828 < 0$.
\item[(2)] All the coefficients of eq. \eqref{en5} are positive.
\item[(3)] $C_1 C_2 > C_3$ implies that $(2.65497)(0.344814) = 0.915471 > 0.151479 $.
\end{enumerate}
Therefore, all the conditions are satisfied. Thus, equilibrium $(E_4)$ is stable.\\
Now we will check the existence of the interior equilibrium point $(E_5)$. After simplification, the equation \eqref{e4} for finding the value of $S_5$ is:
\begin{equation}\label{en6}
S_5^3+ 19.6443 S_5^2+ 62.1861 S_5+ 46.1203=0.
\end{equation}
The roots of eq. \eqref{en6} are $-2.6173, -1.10686$ and $-15.9202$. Thus, equilibrium $(E_5)$ does not exist as none of the roots is positive.\\
Figures \ref{fig:7} and \ref{fig:8} show the results corresponding to the system \eqref{eq33}.
The basic reproduction number $\mathcal{R}_0$ in this case is calculated as $\mathcal{R}_0= 1.95876 >1$. Therefore, disease is endemic.
\end{enumerate}
\section{\textbf{Discussion}} \label{discussion}
The mathematical model which we consider, expressed by four non-linear ordinary differential equations described in \eqref{eq1}.  We have considered this model to study the influence of disease, migration and vaccination on an environment where two or more interacting species are present. The boundedness of the solutions of the system, existence and stability conditions of equilibria are discussed. The model is analyzed with and without infection in prey population. On comparing the disease free model \eqref{eq3} and the main model \eqref{eq1}, we have seen that trivial equilibrium points $(E^{(0)})$ and $(E_0)$ always exist and stable only if conditions \eqref{c1} are satisfied. The dynamic behavior for both the models around trivial equilibrium point is same. Mathematically, we have seen that these equilibria are unstable. If these would be stable, then it tells about the extinction of species in the ecosystem. Now, the equilibrium points $E^{(1)}(S_1,V_1,0)$ and $E_1(S_1,0,V_1,0)$ conveyed the same message ecologically as both represent disease and predator free conditions but at the same time, mathematically they are different as per concern to the dimensional study. Similarly, $E^{(2)}(S_2, V_2, P_2)$ and $E_2(S_2,0,V_2,P_2)$ are disease-free equilibrium points. The existence conditions for the equilibrium points $(E^{(1)})$, $(E_1)$ and $(E^{(2)})$,  $(E_2)$ corresponding to both models are same but the stability conditions of these equilibria are different. If $(E^{(1)})$ and $(E_1)$ are stable, then it simply means prey population will survive for the long period of time as no predation and infection will occur in environment.

Further, we have seen that equilibrium $E_3(0,I_3,0,P_3)$ does not exist as biologically and ecologically, population cannot be assumed negative. The equilibrium point $E_4(S_4,I_4,V_4,0)$ is predator-free equilibrium. The removal of predator has several impacts on prey population like behavioral changes in prey species, etc. The prey population will stay in one place as there are no predators and due to presence of infection, disease can spread more within prey species and that would affect the survival of prey. Mathematically, we observe the increment in the healthy and vaccinated prey as well as the reduction in number of infected prey population. This happens due to the presence of migration and vaccination in prey. Now, the non zero equilibrium $E_5(S_5,I_5,V_5,P_5)$ is the most important equilibrium point as it represents the coexistence of all the species in the ecosystem. This is very essential for the ecological balance.

It has also been observed that equilibrium points are globally asymptotically stable in different 2-D planes. For example, Figure \ref{fig:8} shows that equilibrium $E_4$ is globally asymptotically stable in $S-I$, $S-V$ and $S-P$ planes for the system.

As mentioned in remark \ref{rem1} if $\phi = 0$, then no vaccination takes place. Also, the number of vaccinated prey will increase with the increment in the value of $\phi$. Figure \ref{fig:9} shows the solutions of vaccinated prey for the different values of $\phi$. 
\begin{figure}[h!]
\includegraphics[scale=0.5]{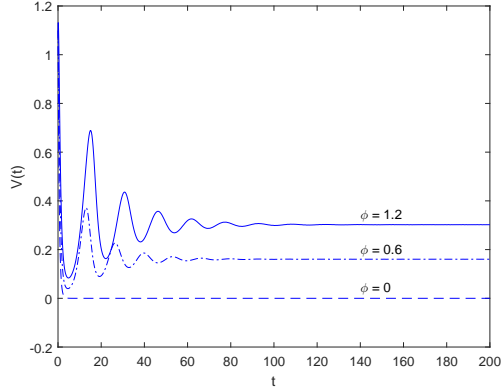}
\caption{Solutions of vaccinated prey population for different values of $\phi$.}
\label{fig:9}
\end{figure}
\par
If the condition $\mathcal{R}_0<1$ is satisfied, then the infection will die out. Generally, it is very difficult to control the epidemic for the larger estimation of $\mathcal{R}_0$. The relation between susceptible prey and infected prey for the different conditions on $\mathcal{R}_0$ is shown in Figure \ref{fig:10}. The change in number of susceptible prey has been observed while increasing the value of $\mathcal{R}_0$.
\begin{figure}[h!]
\subfloat[$\mathcal{R}_0<1$]{\includegraphics[scale=0.36]{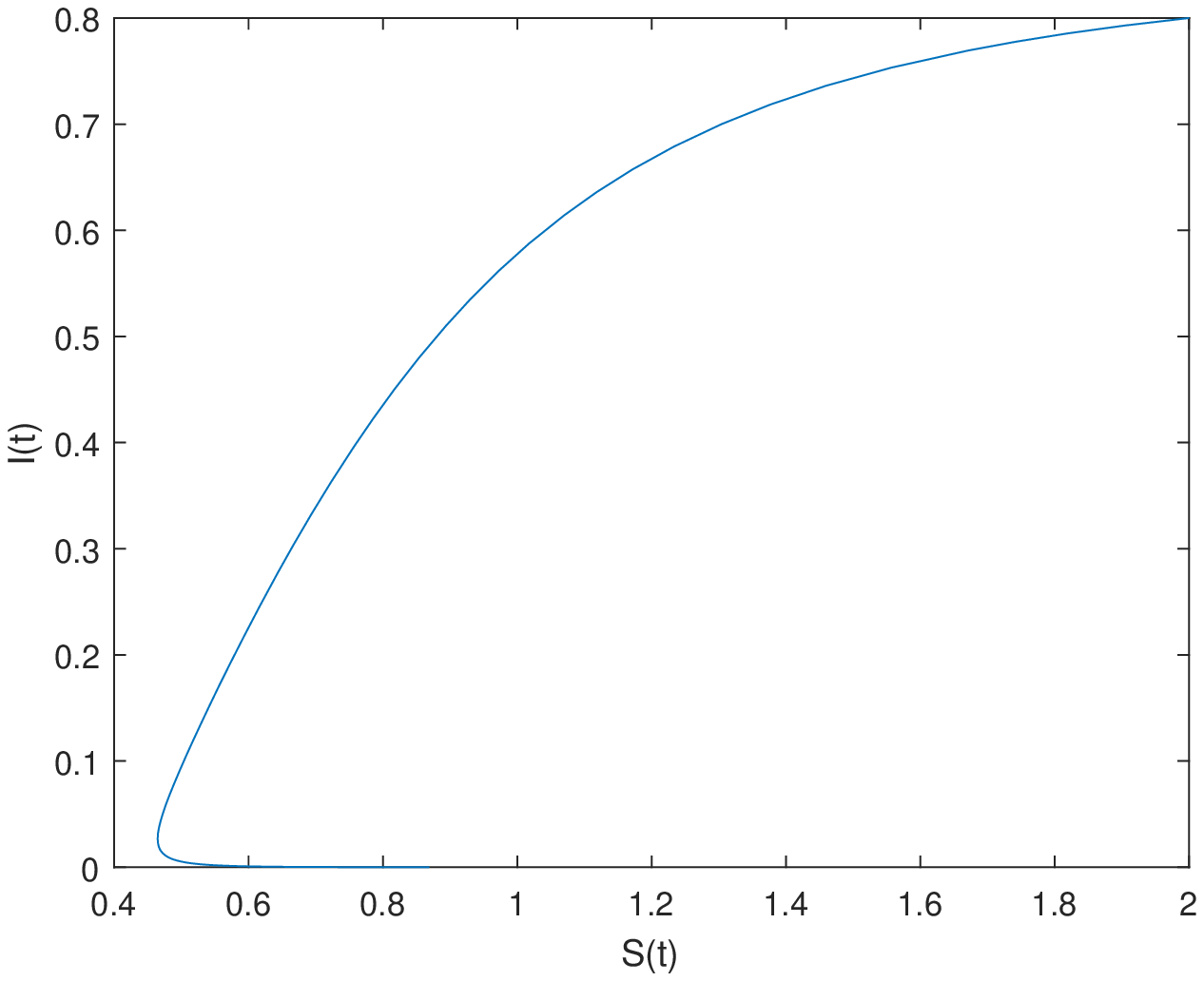}}
\subfloat[$\mathcal{R}_0=1$]{\includegraphics[scale=0.36]{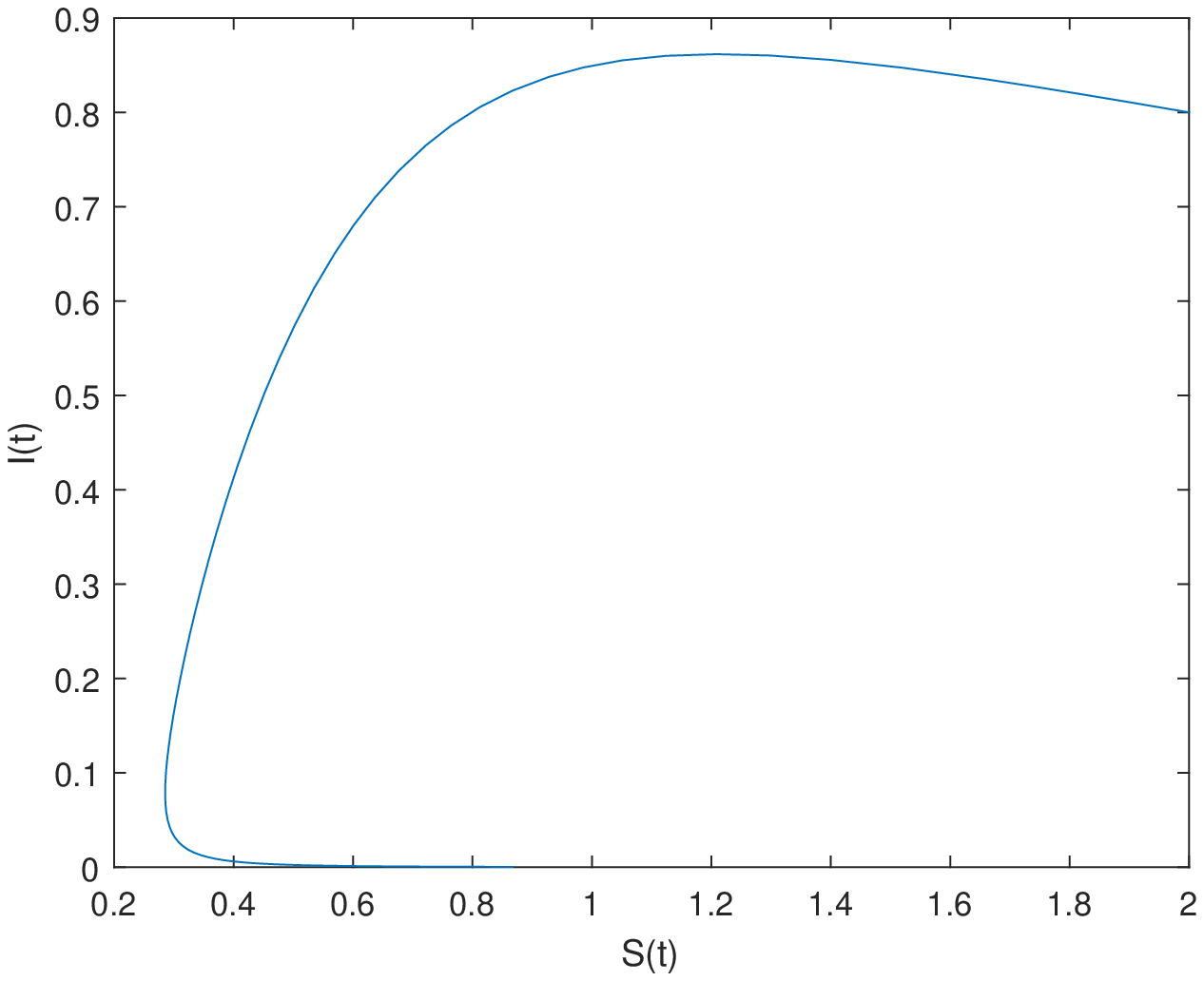}} 
\subfloat[$\mathcal{R}_0>1$]{\includegraphics[scale=0.36]{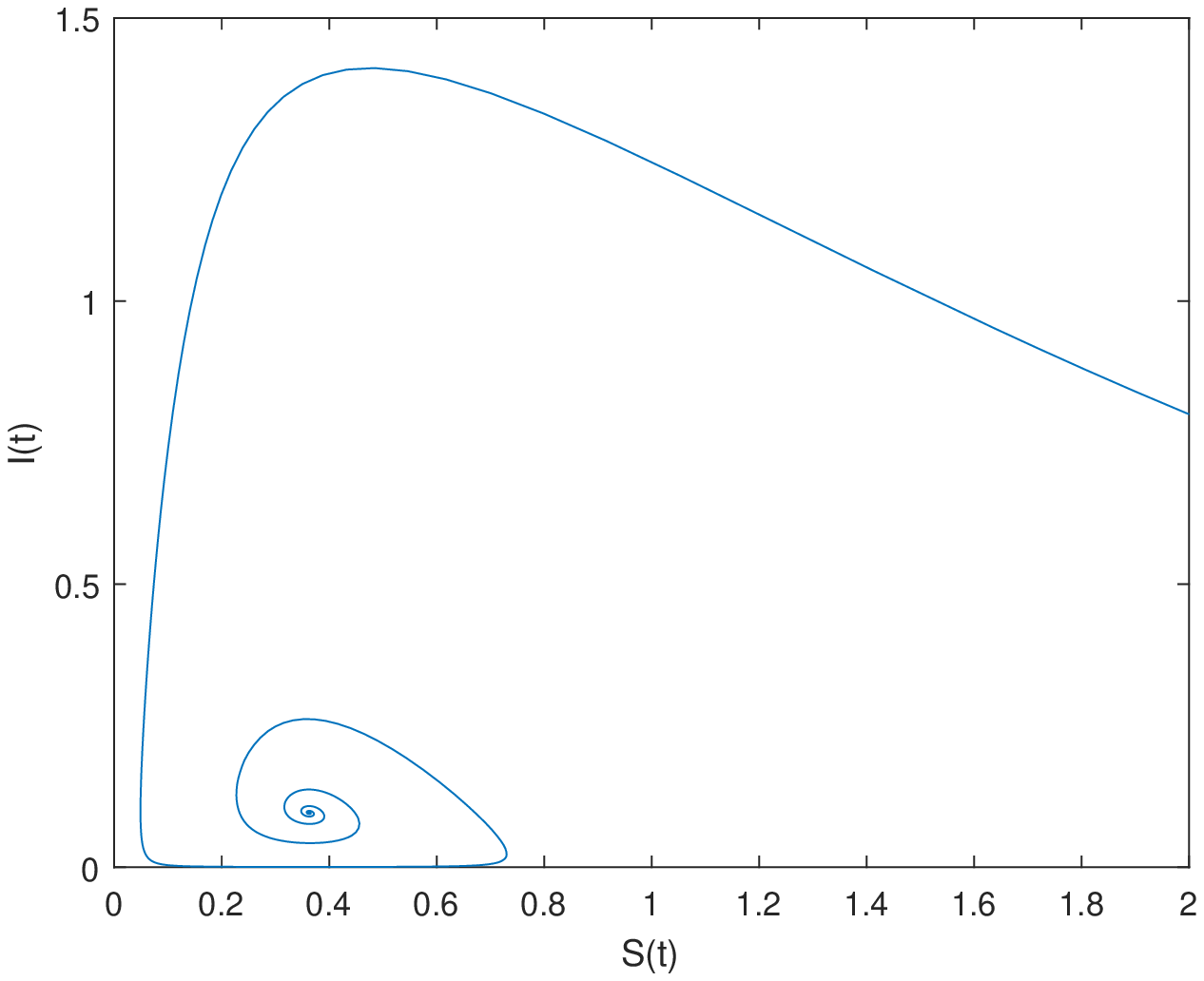}}
\caption{Numerical solutions for different values of $\mathcal{R}_0$.}
\label{fig:10}
\end{figure}
\par 
We stated in remark \ref{rem2} that migration is not same as mortality. Migration plays a different role to explore more about the system \eqref{eq1}. As we have observed the changes in solutions for main model in the absence and presence of migration. Figures \ref{fig:6} and \ref{fig:8} explain that the number of infected prey population decreases when we consider migration in our model. This also implies the increment in the population of healthy prey, it simply means that the infection is reducing within the population. Similarly, it has been observed that the number of vaccinated prey increases with the effect of migration in main model. The dependency of equilibria on migration is described in Table 2.
\begin{table}[h!]
\caption{Dependency of equilibria on migration.}
\begin{tabular}{c@{\hskip 0.6in}l@{\hskip 0.6in} l}
\hline
Equilibria & Existence & Stability\\
\hline
$(E_0)$ & Independent & $m_1$ and $m_3$\\
$(E_1)$ & $m_1$ and $m_3$  & $m_1$, $m_2$ and $m_3$\\
$(E_2)$ & $m_1$ and $m_3$ & $m_1$, $m_2$ and $m_3$\\
$(E_4)$ & $m_1$, $m_2$ and $m_3$ & $m_1$, $m_2$ and $m_3$\\
$(E_5)$ & $m_1$, $m_2$ and $m_3$ & $m_1$, $m_2$ and $m_3$\\
\hline
\end{tabular}
\end{table}

We have seen that the conditions of existence and stability of equilibrium points depend on the parameters. Thus, the estimation of parameters is an important phase for numerical simulations of a mathematical model.
\section*{\textbf{Acknowledgement}}
This research is financially supported by University Grant Commission (UGC), Government of India to the author, Harsha Kharbanda (Sr. No. 2121440663). She gratefully acknowledges the support for the research work.
\appendix
\section{}
\subsection{(Coefficients of Eq.\eqref{eq7})}\label{Appendix:a1} 
The coefficients of equation \eqref{eq7} are given by:
\begin{align*}
B_1 =& -tr(A)= -(\alpha_{11} + \alpha_{22}+\alpha_{33})\\
=& -\left(r-\frac{2 r S_3}{k} -\phi -p_1 P_3 - m_1-d_1-\theta - p_3 P_3 -m_3-d_3 +q_1 p_1 S_3+ q_3 p_3 V_3-d_4\right).\\
B_2 =& (\alpha_{11}\alpha_{22} - \alpha_{12}\alpha_{21})+(\alpha_{11}\alpha_{33}-\alpha_{13}\alpha_{31})+(\alpha_{22}\alpha_{33}-\alpha_{23}\alpha_{32})\\
=& [(r-\frac{2 r S_3}{k} -\phi -p_1 P_3 - m_1-d_1)(-\theta - p_3 P_3 -m_3-d_3)-\theta \phi]+ \\  
& [(r-\frac{2 r S_3}{k} -\phi -p_1 P_3 - m_1-d_1)(q_1 p_1 S_3+ q_3 p_3 V_3-d_4)-(-p_1 S_3)(q_1 p_1 P_3)]+\\
& [( -\theta - p_3 P_3 -m_3-d_3)(q_1 p_1 S_3+ q_3 p_3 V_3-d_4)-( -p_3 V_3)(q_3 p_3 P_3)]\\
= & \theta m_1 +d_3 m_1 + \phi m_3 + m_1 m_3 - \frac{(-2 r S_3 + k (r - \phi)) d_3 + r (k - 2 S_3) (\theta + d_4 + m_3)}{k}+ P_3 \theta p_1 \\
&+ P_3 d_3 p_1 +  P_3 m_3 p_1 - P_3 r p_3 + P_3 \phi p_3 +P_3 m_1 p_3 +P_3^2 p_1 p_3+ d_4 (\theta+ \phi + d_3+ m_1+ m_3 \\
&+ P_3 (p_1 + p_3))+r S_3 p_1 q_1 - S_3 \theta p_1 q_1 + S_3 \phi p_1 q_1- S_3 d_3 p_1 q_1 - S_3 m_1 p_1 q_1-S_3 m_3 p_1 q_1 \\
& -P_3 S_3 p_1 p_3 q_1 +d_1 (\theta + d_3 + d_4 + m_3 -S_3 p_1 q_1 +p_3(P_3 - V_3 q_3))\\
&+\frac{2 P_3 r S_3 p_3- 2 r S_3^2 p_1 q_1-V _3(2 r S_3 + k (-r + \theta + \phi) + k (d_3 + m_1 + m_3 +P_3 p_1)) p_3 q_3}{k}.\\
B_3 =& -det(A) = -[\alpha_{11}(\alpha_{22} \alpha_{33} -\alpha_{32} \alpha_{23})+\alpha_{12}(\alpha_{31} \alpha_{23}-\alpha_{33} \alpha_{21}) + \alpha_{13}(\alpha_{21}\alpha_{32}-\alpha_{31}\alpha_{22})]\\
\end{align*}
\begin{align*}
=& -\frac{1}{k}[d_4(r (k - 2 S_3) \theta - k \theta m_1 +k r m_3 - 2 r S_3 m_3 -k \phi m_3 - k m_1 m_3 -k P_3 \theta p_1 -k P_3 m_3 p_1 +\\
& d_3(k r - 2 r S_3 - k \phi -k (m_1 + P_3 p_1)) -P_3 (-k r + 2 r S_3 + k \phi + k m_1 +k P_3 p_1) p_3 -k d_1\\
& (\theta + d_3 + m_3 + P_3 p_3)) +V (-r (k - 2 S_3) \theta + k \theta m_1 +d_3 (-k r + 2 r S_3 + k \phi + k m_1) \\&
+(-k r +2 r S_3+ k \phi + k m_1) m_3 + k d_1 (\theta + d_3 + m_3)) p_3 q_3 +p_1 ((S_3 (-r (k - 2 S_3) \theta + k \theta m_1 \\
&+d_3 (-k r + 2 r S_3 + k \phi + k m_1) + (-k r + 2 r S_3 + k \phi +k m_1) m_3 +k d_1 (\theta + d_3 + m_3)) + P_3\\
& (-k r S_3 + 2 r S_3^2 - k V_3 \theta + k S_3 \phi +k S_3 (d_1 + m_1)) p_3) q_1 +k P_3 (V_3 \theta - S_3 \phi + V_3 (d_3 + m_3)) p_3 q_3)]
\end{align*}
\subsection{(Coefficients of Eq.\eqref{eq8})} \label{Appendix:a2}
After simplification, the coefficients of equation \eqref{eq8} are:
\begin{align*}
C_1 =& -tr(B) = -(\kappa_{11} + \kappa_{22} + \kappa_{33}),\\
=& -\left(r-2r\frac{S_4}{k}-\frac{rI_4}{k}-\beta I_4 -\phi - m_1-d_1+\beta S_4 + \sigma V_4  - m_2-d_2-c-\sigma I_4 -\theta -m_3-d_3 \right).\\
C_2 =& (\kappa_{11}\kappa_{22} - \kappa_{12}\kappa_{21})+(\kappa_{11}\kappa_{33}-\kappa_{13}\kappa_{31})+(\kappa_{22}\kappa_{33}-\kappa_{23}\kappa_{32}),\\
=& \frac{1}{k}[r (-2\beta S_4^2 +2 \theta S_4 + \theta I_4 + (2 S_4 + I_4)(-V_4 + I_4) \sigma) -k (-\beta \theta I_4 + V_4 \theta \sigma +V_4 \beta I_4 \sigma  - \beta I_4^2 \sigma\\
&+r (-S_4 \beta+ \theta - V_4 \sigma + I_4 \sigma) +V_4 \sigma \phi - I_4 \sigma \phi +S_4 \beta (\theta + I_4 \sigma + \phi)) + c (r (2 S_4 + I_4) +k(-r + \theta \\&
+ I_4 (\beta + \sigma) + \phi)) + c k d_3 - k r d_3 + 2 r S_4 d_3 -k S_4 \beta d_3 + r I_4 d_3 + k \beta I_4 d_3 - k V_4 \sigma d_3+ k \phi d_3 +c k m_1 \\&
- k S_4 \beta m_1 + k \theta m_1 - k V_4 \sigma m_1 + k I_4 \sigma m_1 + k d_3 m_1 - k r m_2 + 2 r S_4 m_2+ k \theta m_2 + r I_4 m_2 + k \beta I_4 m_2\\&
 + k I_4 \sigma m_2 +k \phi m_2 + k d_3 m_2 +k m_1 m_2 + (c k + r (2 S_4 + I_4)-k (r + S_4 \beta - \beta I_4 + V_4 \sigma - \phi) \\&
+k (m_1 + m_2)) m_3 +k d_1 (c - S_4 \beta + \theta -V_4 \sigma + I_4 \sigma + d_2+d_3 + m_2 + m_3) +d_2 (r (2 S_4 + I_4) \\&
+k (-r + \theta + I_4 (\beta + \sigma) + \phi) +k (d_3 + m_1 + m_3))]\\
C_3 =& -det(B) = -[\kappa_{11}(\kappa_{22} \kappa_{33} -\kappa_{32} \kappa_{23})+\kappa_{12}(\kappa_{31} \kappa_{23}-\kappa_{33} \kappa_{21}) + \kappa_{13}(\kappa_{21}\kappa_{32}-\kappa_{31}\kappa_{22})]\\
=& \frac{1}{k}[-c (r (k - 2 S_4) - (r +k \beta) I_4) (\theta + I_4 \sigma) + c k I_4 \sigma \phi +r (k V_4 \theta \sigma - V_4 \theta I_4 \sigma +k S_4 \beta (\theta + I_4 \sigma)-2 S_4^2 \beta \\
&(\theta + I_4 \sigma) +S_4 \sigma (-2 V_4 \theta + I_4 \phi)) - c k r d_3 + 2 c r S_4 d_3 + k r S_4 \beta d_3 - 2 r S_4^2 \beta d_3 +c r I_4 d_3 +c k \beta I_4 d_3 \\
&+ k r V_4 \sigma d_3 - 2 r S_4 V_4 \sigma d_3 - r V_4 I_4 \sigma d_3 -k V_4 \beta I_4 \sigma d_3 +c k \phi d_3 - k S_4 \beta \phi d_3 - k V_4 \sigma \phi d_3 + c k \theta m_1\\
& +c k d_3 m_1-k S_4 \beta \theta m_1 -k V_4 \theta \sigma m_1 +c k I_4 \sigma m_1 -k S_4 \beta I_4 \sigma m_1 -k S_4 \beta d_3 m_1 -k V_4 \sigma d_3 m_1 -k r \theta m_2\\
& + 2 r S \theta m_2 +r \theta I_4 m_2 +k \beta \theta I_4 m_2 -k r I_4 \sigma m_2 +2 r S_4 I_4 \sigma m_2 + r I_4^2 \sigma m_2 + k \beta I_4^2 \sigma m_2 + k I_4 \sigma \phi m_2 \\
&- k r d_3 m_2 + 2 r S_4 d_3 m_2 + r I_4 d_3 m_2 + k \beta I_4 d_3 m_2 + k \phi d_3 m_2 + k \theta m_1 m_2 + k I_4 \sigma m_1 m_2 + k d_3 m_1 m_2 \\
&+ (-2 r S_4^2 \beta + c r (2 S_4 + I_4) -r V_4 (2 S_4 + I_4) \sigma +k r (S_4 \beta + V_4 \sigma) - k S_4 \beta \phi - k V_4 \sigma (\beta I_4 + \phi) + c k\\
& (-r + \beta I_4 + \phi) + (r (2 S_4 +I_4) + k (-r + \beta I_4 + \phi)) m_2 + k m_1 (c - S_4 \beta - V_4 \sigma + m_2)) m_3 + d_2\\
& (-(r (k - 2 S_4) - (r +k \beta) I_4) (\theta + I_4 \sigma) + k I_4 \sigma \phi +d_3(r(2 S_4 + I_4) +k (-r + \beta I_4 + \phi) +k m_1)  \\
&+ (r (2 S_4 + I_4) + k (-r + \beta I_4 + \phi)) m_3 +k m_1 (\theta + I_4 \sigma+ m_3)) +k d_1 (-V_4 \theta \sigma + c (\theta + I_4 \sigma) \\
& -S \beta (\theta + I_4 \sigma) + (\theta + I_4 \sigma) m_2 +d_3 (-V_4 \sigma + m_2) +(-V \sigma + m_2) m_3+ (c - S_4 \beta) (d_3 + m_3) +d_2\\
& (\theta + I_4 \sigma + d_3 +m_3))].
\end{align*} 
\subsection{(Coefficients of Eq.\eqref{eq9})}
After simplification, the coefficients of equation \eqref{eq9} are
\begin{align*}
D_1 =& -tr(J(E_5))\\
=& - (c_{11}+c_{22}+c_{33}+c_{44})\\
D_2 =& \text{Sum of all the possible second order principal minors}\\
=& c_{33} c_{44} + c_{22} (c_{33} + c_{44}) + c_{11} (c_{22} + c_{33} + c_{44}) + \frac{(r + k \beta)\beta S_5 I_5}{k} +\sigma^2  V_5 I_5- \theta \phi + P_5 (p_1^2 q_1 S_5\\
&+ p_2^2 q_2 I_5+p_3^2 q_3 V_5)\\
D_3 =& -(\text{Sum of all the possible third order principal minors})\\
=& \frac{1}{k} [-(c_{33} + c_{44})(r + k \beta) \beta S_5 I_5 +k \beta \theta V_5 I_5 \sigma - c_{44} k V_5 I_5 \sigma^2 - k c_{11} (c_{33} c_{44} + c_{22} (c_{33} + c_{44})+\sigma^2 V_5 I_5 ) \\
&+ k \theta \phi c_{44}+ (r + k \beta)\sigma \phi S_5 I_5 +c_{22} k (-c_{33} c_{44} + \theta \phi) +P_5 \{-p_1 ((c_{22} + c_{33}) k S_5 p_1 +(r + k \beta)S_5 I_5 p_2\\
& -k V_5 \theta p_3) q_1 +k  p_2 I_5(S_5 \beta p_1 - (c_{11} + c_{33}) p_2 + \sigma p_3 V_5) q_2 -k p_3 q_3 (-\phi p_1 S_5+ \sigma p_2 V_5 I_5 \\
&+ (c_{11} + c_{22}) p_3 V_5)\}]\\
D_4 =& det(J(E_5))\\
=& \frac{1}{k}[c_{44} (c_{33} \beta S_5 I_5(r + k \beta) -k \beta \theta \sigma  I_5 V_5 + k c_{11} (c_{22} c_{33} + V_5 I_5 \sigma^2) - (k \theta c_{22} +S_5 (r + k \beta) I_5 \sigma) \phi) \\
&+P_5 \{k p_1^2 q_1 S_5 (c_{22} c_{33} + V_5 I_5 \sigma^2) + k p_2 q_2 I_5 ((c_{11} c_{33} - \theta \phi) p_2 + p_3 V_5 (\beta \theta -\sigma c_{11})) + (p_2 I_5(c_{11}k V_5 \sigma\\
& -\phi S_5 (r + k \beta)) +p_3 V_5(c_{11}  c_{22} k + S_5 \beta (r + k \beta) I_5))p_3 q_3 +p_1 (I_5 p_2 (q_1(c_{33} S_5 (r + k \beta)- k V_5 \theta \sigma)\\
& + k q_2 S_5 (-c_{33} \beta + \sigma \phi))+p_3 (-q_1 V_5 (c_{22} k \theta +S_5 (r + k \beta) I_5 \sigma)- k q_3 S_5 (V_5 \beta I_5 \sigma + c_{22} \phi)))\}]
\end{align*}
\label{Appendix:a3}

\end{document}